\documentclass{amsart}
\pdfoutput=1 
\usepackage{amssymb, latexsym, amsmath, verbatim, amsthm,
  amscd,stmaryrd,xspace, xfrac}
\usepackage{pinlabel}
\usepackage{hyperref}

\title{Growth Tight Actions}
\author{Goulnara N. Arzhantseva}
\email{\href{mailto:goulnara.arzhantseva@univie.ac.at}{goulnara.arzhantseva@univie.ac.at}}
\address{Fakult\"at f\"ur Mathematik\\
Universit\"at Wien}
\author{Christopher H. Cashen}
\email{\href{mailto:christopher.cashen@univie.ac.at}{christopher.cashen@univie.ac.at}}
\address{Fakult\"at f\"ur Mathematik\\
Universit\"at Wien}
\author{Jing Tao} 
\email{\href{mailto:jing@ou.edu}{jing@ou.edu}}
\address{Department of Mathematics\\University of Oklahoma}
\date{March 9, 2015}

\thanks{The first two authors supported by the European Research
  Council (ERC) grant of Goulnara ARZHANTSEVA, grant agreement
  \#259527 and the Erwin
Schr\"odinger Institute workshop ``Geometry of Computation in
Groups''. The second author partially supported by the
Austrian Science Fund (FWF):M1717-N25. The third author supported by NSF grant DMS-1311834.}

\hypersetup{
    pdftitle={Growth Tight Actions},    
    pdfauthor={Goulnara Arzhantseva, Christopher H Cashen, and Jing Tao},     
    pdfkeywords={Growth tight, Hopf property, entropy, relatively hyperbolic groups, mapping class groups, snowflake groups, rank 1 isometries, CAT(0) groups}, 
    colorlinks=true,       
    linkcolor=black,          
    citecolor=black,        
    filecolor=black,      
    urlcolor=black           
}

\theoremstyle{plain}
\newtheorem{theorem}{Theorem}[section]
\newtheorem{lemma}{Lemma}[section]
\newtheorem{proposition}{Proposition}[section]
\newtheorem{corollary}{Corollary}[section]
\newtheorem*{reftheorem}{Theorem}

\newtheorem{question}{Question}

\theoremstyle{remark}
\newtheorem{claim}{Claim}[theorem]
\newtheorem{remark}{Remark}[section]

\theoremstyle{definition}
\newtheorem{definition}{Definition}[section]

\newtheorem{observation}{Observation}[section]

\def\makeautorefname#1#2{\expandafter\def\csname#1autorefname\endcsname{#2}}
\let\fullref\autoref

\makeautorefname{theorem}{Theorem} 
\makeautorefname{lemma}{Lemma} 
\makeautorefname{proposition}{Proposition} 
\makeautorefname{corollary}{Corollary} 
\makeautorefname{definition}{Definition}
\makeautorefname{example}{Example}
\makeautorefname{section}{Section}
\makeautorefname{subsection}{Section}
\makeautorefname{subsubsection}{Section}
\makeautorefname{claim}{Claim}
\makeautorefname{observation}{Observation}
\makeautorefname{remark}{Remark}

\makeatletter 
\let\c@lemma=\c@theorem 
\makeatother
\makeatletter 
\let\c@proposition=\c@theorem 
\makeatother
\makeatletter 
\let\c@corollary=\c@theorem 
\makeatother
\makeatletter 
\let\c@definition=\c@theorem 
\makeatother
\makeatletter 
\let\c@example=\c@theorem 
\makeatother
\makeatletter 
\let\c@observation=\c@theorem
\makeatother
\makeatletter 
\let\c@remark=\c@theorem
\makeatother
\makeatletter
\@addtoreset{claim}{proposition}
\makeatother
\makeatletter
\@addtoreset{claim}{lemma}
\makeatother

\newcommand{\ccpp}{coarsely a closest point projection\xspace}
\newcommand{\ball}[2]{\mathcal{B}_{#1}(#2)}
\newcommand{\clball}[2]{\closure{\mathcal{B}}_{#1}(#2)}
\newcommand{\nbhd}[2]{\ball{#1}{#2}}
\newcommand{\clnbhd}[2]{\closure{\mathcal{B}}_{#1}(#2)}
\newcommand{\thick}{\clnbhd{Q}{G.\bp}}
\newcommand{\thin}{\X\setminus\thick}
\newcommand{\gpcenter}{Z}

\def\act{\curvearrowright} 
\newcommand{\closure}[1]{\overline{#1}} 

\makeatletter
\newsavebox\myboxA
\newsavebox\myboxB
\newlength\mylenA

\newcommand*\xoverline[2][0.75]{%
    \sbox{\myboxA}{$\m@th#2$}%
    \setbox\myboxB\null
    \ht\myboxB=\ht\myboxA%
    \dp\myboxB=\dp\myboxA%
    \wd\myboxB=#1\wd\myboxA
    \sbox\myboxB{$\m@th\overline{\copy\myboxB}$}
    \setlength\mylenA{\the\wd\myboxA}
    \addtolength\mylenA{-\the\wd\myboxB}%
    \ifdim\wd\myboxB<\wd\myboxA%
       \rlap{\hskip 0.5\mylenA\usebox\myboxB}{\usebox\myboxA}%
    \else
        \hskip -0.5\mylenA\rlap{\usebox\myboxA}{\hskip 0.5\mylenA\usebox\myboxB}%
    \fi}
\makeatother
\newcommand{\inv}[1]{{#1}^{-1}}

\def\from{\colon\thinspace}
\def\onto{\twoheadrightarrow}

\newcommand{\constm}{M}
\newcommand{\consta}{C}
\newcommand{\constb}{K}
\newcommand{\constc}{D}
\newcommand{\rate}{\delta}
\newcommand{\bottleneck}{\Delta}

\newcommand{\geo}{\mathcal{L}}
\newcommand{\qgeo}{\mathcal{Q}}
\newcommand{\vor}{\mathcal{V}}
\newcommand{\zuruck}{\mathcal{Z}}
\def\H{\mathcal{H}}
\def\A{\mathcal{A}}
\def\bp{o}
\def\qtbp{\star}
\DeclareMathOperator{\diam}{diam}
\def\ps{\mathcal{PS}}
\def\path{\mathcal{P}}
\def\G{\mathcal{G}}

\def\qt{\mathcal{Y}}
\def\X{\mathcal{X}}
\def\Y{\mathcal{Y}}
\def\G{\mathcal{G}}
\def\pers{\underline{\mathcal{P}}}
\def\per{\mathcal{P}}
\def\parabs{\underline{P}}
\def\parab{P}

  \newcommand{\emul}{\stackrel{{}_\ast}{\asymp}}
  \newcommand{\gmul}{\stackrel{{}_\ast}{\succ}}
  \newcommand{\lmul}{\stackrel{{}_\ast}{\prec}}
  \newcommand{\eadd}{\stackrel{{}_+}{\asymp}}
  
  \newcommand{\ladd}{\stackrel{{}_+}{\prec}}
  \newcommand{\laddmul}{\prec}
  \newcommand{\eaddmul}{\asymp}
  \newcommand{\gaddmul}{\succ}

  \newcommand{\s}{\ensuremath{\mathcal{S}}\xspace} 
   
  \newcommand{\mcg}{\ensuremath{\operatorname{Mod}(\s)}\xspace}  
   
  \newcommand{\param}{{\mathchoice{\mkern1mu\mbox{\raise2.2pt\hbox{$
  \centerdot$}}
  \mkern1mu}{\mkern1mu\mbox{\raise2.2pt\hbox{$\centerdot$}}\mkern1mu}{
  \mkern1.5mu\centerdot\mkern1.5mu}{\mkern1.5mu\centerdot\mkern1.5mu}}}

\begin{document}
\begin{abstract}
We introduce and systematically study the concept of a growth tight action. 
This generalizes growth tightness for word metrics as initiated by Grigorchuk and de la Harpe.
Given a finitely generated, non-elementary group $G$ 
acting on a $G$--space $\mathcal{X}$, we prove that
if $G$ contains a strongly contracting element and if $G$ is not too
badly distorted in $\mathcal{X}$, then the action of $G$ on $\mathcal{X}$ is a growth tight action.
It follows that  if $\mathcal{X}$ is a cocompact, relatively hyperbolic $G$--space, then the action of $G$ on $\mathcal{X}$ is a growth tight action.
This generalizes all previously known results for growth tightness of cocompact actions: every already known example of
a group that admits a growth tight action and has some infinite, infinite index
normal subgroups is relatively hyperbolic, and, conversely, relatively hyperbolic
groups admit growth tight actions. 
This also allows us to prove that many CAT(0) groups, including
flip-graph-manifold groups and many Right Angled Artin Groups, and
snowflake groups admit cocompact, growth tight actions. 
These provide first examples of non relatively hyperbolic groups
admitting interesting growth tight actions.
Our main result applies  as well to cusp uniform actions on hyperbolic
spaces and to the action of the mapping class group on Teichm\"uller
space with the Teichm\"uller metric.
Towards the proof of our main result, we give equivalent characterizations of strongly contracting elements
and produce new examples of group actions with strongly contracting elements.

\end{abstract}

\maketitle
\setcounter{tocdepth}{1}
\tableofcontents


\setcounter{section}{-1}
\section{Introduction}

The growth exponent of a set $\A\subset\X$ with respect to a
pseudo-metric $d$ is
\[\rate_{\A,d} :=\limsup_{r\to\infty}\frac{\log \# \{a\in \A\mid d(\bp,a)\leqslant r\}}{r}\]
where $\#$ denotes cardinality and $\bp\in \X$ is some basepoint. The limit is independent of the choice of basepoint.

Let $G$ be a finitely generated group.
A left invariant pseudo-metric $d$ on $G$ induces a left invariant
pseudo-metric $\bar{d}$ on any quotient $G/\Gamma$ of $G$ by
$\bar{d}(g\Gamma,g'\Gamma):=d(g\Gamma,g'\Gamma)$.
    \begin{definition}
      $G$ is \emph{growth tight} with respect to $d$ if
      $\rate_{G, d}>\rate_{G/\Gamma, \bar{d}}$ for every infinite
      normal subgroup $\Gamma\trianglelefteqslant G$.
    \end{definition}

One natural way to put a left invariant metric on a finitely generated
group is to choose a finite generating set and consider the word
metric.
More generally, pseudo-metrics on a group are provided by actions of the
group on metric spaces.
Let $\X$ be a $G$--space, that is, a proper, geodesic metric space with a properly
discontinuous, isometric $G$--action $G\act \X$.
The choice of a basepoint $\bp\in \X$ induces a left invariant
pseudo-metric on $G$ by $d_G(g,g'):=d_\X(g.\bp,g'\hspace{-.2em}.\bp)$.

Define the growth exponent $\rate_G$ of $G$ with respect
to $\X$ to be
the growth exponent of $G$ with respect to an induced pseudo-metric $d_G$.
This depends only on the $G$--space $\X$, since a different choice of
basepoint in $\X$
defines a pseudo-metric that differs from $d_G$ by an additive
constant. 
Likewise, let $\rate_{G/\Gamma}$ denote the growth exponent of $G/\Gamma$
with respect to a pseudo-metric on $G/\Gamma$ induced by $d_\X$.

\begin{definition}
  $G\act \X$ is a \emph{growth tight action} if
  $\rate_G>\rate_{G/\Gamma}$ for every infinite normal subgroup
  $\Gamma\trianglelefteqslant G$.
\end{definition}

Some groups admit growth tight actions for the simple reason that they
lack any infinite, infinite index normal subgroups. 
For such a group $G$, every action on a $G$--space with positive
growth exponent will be growth tight.
Exponentially growing simple groups are examples, as, by the Margulis
Normal Subgroup Theorem \cite{Mar91},
 are irreducible
lattices in higher rank semi-simple Lie groups.

Growth tightness\footnote{Grigorchuk and de la Harpe define growth
  tightness in terms of `growth rate',
which is just the exponentiation of our growth exponent. The growth exponent definition is analogous to the notion of
`volume entropy' familiar in Riemannian geometry, and is more
compatible with the Poincar\'e series in \fullref{sec:poincare}.} for
word metrics was introduced and studied by Grigorchuk and de la
Harpe \cite{GriDeL97}, who showed, for example, that a finite rank free
group equipped with the word metric from a free generating set is
growth tight. 
On the other hand, they showed that the product of a free group with itself,
generated by free generating sets of the factors, is not growth tight.
Together with the Normal Subgroup Theorem, these results suggest that
for interesting examples of growth tightness we should examine `rank
1' type behavior.
Further evidence for this idea comes from the work of Sambusetti and
collaborators, who in a series of papers \cite{Sam02surface, Sam03,
  Sam04,DalPeiPic11} prove growth tightness for
the action of the fundamental group of a negatively curved Riemannian
manifold on its Riemannian universal cover.

In the study of non-positively curved, or CAT(0), spaces there is a
well established idea that a space may be non-positively curved but
have some specific directions that look negatively curved.
More precisely:
\begin{definition}[\cite{BalBri95}]
A hyperbolic isometry of a proper CAT(0) space is \emph{rank 1} if
it has an axis that does not bound a half-flat.
\end{definition}

In \fullref{def:contractinglement}, we introduce the notion for an
element of $G$ to be \emph{strongly contracting} with respect to 
$G\act \X$.
In the case that $\X$ is a CAT(0) $G$--space, the strongly contracting
elements of $G$ are precisely those that act as rank 1 isometries of $\X$ (see \fullref{Thm:rank1}).

In addition to having a strongly contracting element, we will 
assume that the orbit of $G$ in $\X$ is not too badly distorted.
There are two different ways to make this precise.

We say a $G$--space is $\consta$--\emph{quasi-convex} if there exists a
$\consta$--quasi-convex $G$--orbit (see \fullref{def:quasiconvex} and \fullref{def:qcgspace}).
This means that it is possible to travel along geodesics joining
points in the orbit of $G$ without leaving a neighborhood of the
orbit.

\begin{reftheorem}[(\fullref{mainqc})]
Let $G$ be a finitely generated, non-elementary group.
Let $\X$ be a quasi-convex $G$--space.
If $G$ contains a strongly contracting element
then $G\act \X$ is a growth tight action.
\end{reftheorem}

Alternatively, we can assume that the growth rate of the number of
orbit points that can be reached by geodesics lying entirely, except near the endpoints, outside a
neighborhood of the orbit is strictly smaller than the growth rate of
the group:

\begin{reftheorem}[(\fullref{maincge})]
Let $G$ be a finitely generated, non-elementary group.
  Let $\X$ be a $G$--space.
If $G$ contains a strongly contracting element and there exists a
$\consta\geqslant 0$ such that the $\consta$--complementary growth exponent of $G$ is strictly less than the
  growth exponent of $G$, then $G\act \X$ is a growth tight action.
\end{reftheorem}
See \fullref{def:cge} for a precise definition of  the $\consta$--complementary growth exponent.

The proof of \fullref{mainqc} is a special case of the proof of \fullref{maincge}.

Using \fullref{mainqc}, we prove:
\begin{reftheorem}[(\fullref{correlativelyhyperbolic})]
 If $\X$ is a quasi-convex, relatively hyperbolic $G$--space and $G$
  does not coarsely fix a peripheral subspace then $G\act \X$ is a
  growth tight action.
\end{reftheorem}

This generalizes all previously known results for growth tightness of
cocompact actions:
every already known example of a group that admits a growth tight action and
has some infinite, infinite index normal subgroups is relatively
hyperbolic, and, conversely, relatively hyperbolic groups admit growth
tight
actions~\cite{ArzLys02,Sam02amalgam,Yan13,Sam03,Sab13,DalPeiPic11}.

We also use \fullref{mainqc} to prove growth tightness for actions on
non relatively hyperbolic spaces. 
For instance, we prove that a group action on a proper CAT(0) space
with a rank 1 isometry is growth tight:
\begin{reftheorem}[(\fullref{cat0})]
  If $G$ is a finitely generated, non-elementary group and $\X$ is a quasi-convex, CAT(0)
  $G$--space such that $G$ contains
an element that acts as a rank 1 isometry on $\X$, then $G\act \X$ is a growth tight action.
\end{reftheorem}

Two interesting classes of non relatively hyperbolic groups to which
\fullref{cat0} applies are non-elementary Right Angled Artin Groups, which are
non relatively hyperbolic when the defining graph is connected, and
flip-graph-manifolds. 
These are the first examples of non relatively hyperbolic groups that
admit non-trivial growth tight actions.
\begin{reftheorem}[(\fullref{raag})]
  Let $\Theta$ be a finite graph that is not a join and has more than
one vertex.
The action of the Right Angled Artin Group $G$ defined by $\Theta$ on
the universal cover $\X$ of the
Salvetti complex associated to $\Theta$ is a growth tight action.
\end{reftheorem}
\begin{reftheorem}[(\fullref{flip})]
    Let $M$ be a flip-graph-manifold. 
Let $G$ and $\X$ be the fundamental group and universal cover,
respectively, of $M$. Then the action of $G$ on $\X$ by deck
transformations is a growth tight action.
\end{reftheorem}

We even exhibit an infinite family of non relatively hyperbolic,
non-CAT(0) groups that admit cocompact, growth tight actions:
\begin{reftheorem}[(\fullref{snowflakegrowthtight})]
  The Brady-Bridson snowflake groups $BB(1,r)$ for $r\geqslant 3$ admit cocompact,
  growth tight actions.
\end{reftheorem}

\bigskip
We prove growth tightness for interesting non-quasi-convex
actions using \fullref{maincge}.
We generalize a theorem of Dal'bo, Peign\'e, Picaud, and
Sambusetti \cite{DalPeiPic11} for Kleinian groups satisfying an
additional Parabolic Gap Condition, see \fullref{def:pgc}, to
cusp-uniform actions on arbitrary hyperbolic spaces satisfying the
Parabolic Gap Condition:
\begin{reftheorem}[(\fullref{theoremcuspuniform})]
 Let $G$ be a finitely generated, non-elementary group.
  Let $G\act \X$ be a cusp uniform action on a hyperbolic space. 
Suppose that $G$ satisfies the Parabolic
Gap Condition.
Then $G\act \X$ is a growth tight action.
\end{reftheorem}

Once again, our theorems extend beyond actions on relatively
hyperbolic spaces, as we use \fullref{maincge} to prove:
\begin{reftheorem}[(\fullref{mcgisgrowthtight})]
  The action of the mapping class group of a hyperbolic surface on its
  Teichm\"uller space with
  the Teichm\"uller metric is a growth tight action. 
\end{reftheorem}

Mapping class groups, barring exceptional low complexity cases, are neither relatively hyperbolic nor CAT(0).

\bigskip
In \fullref{part:gt} of this paper we prove our main results, \fullref{maincge} and \fullref{mainqc}.
We show in \fullref{lemma:lotsofcontractingelements} that if there exists a strongly contracting element for $G\act
\X$ then every infinite normal subgroup $\Gamma$ contains a strongly
contracting element $h$.
We prove growth tightness by bounding the growth exponent of a subset
that is orthogonal, in a coarse sense, to every translate of an axis
for $h$. 

A dual problem, which is of independent interest, is to find the
growth exponent of the conjugacy class of $h$. 
In \fullref{sec:cc} we show that the growth exponent of the
conjugacy class of a strongly contracting element is exactly half the
growth exponent of the group, provided the strongly contracting
element moves the base point far enough.

In \fullref{part:sc} we produce new examples of group actions with
strongly contracting elements.
These include groups acting on relatively hyperbolic metric spaces (\fullref{sec:relhyp}), certain CAT(0) groups (\fullref{sec:cat0}),
mapping class groups (\fullref{sec:mcg}),
and  snowflake groups (\fullref{sec:snowflake}).
Our main theorems imply that all these groups admit growth tight actions.
These are first examples of growth tight actions and groups
which do not come from and are not relatively hyperbolic groups.

\subsection{Invariance}
Growth tightness is a delicate condition.
A construction of Dal'bo, Otal, and Peign\'e \cite{DalOtaPei00}, see \fullref{obs:notqi}, shows
that there exist groups $G$ and non-cocompact, hyperbolic,
equivariantly quasi-isometric $G$--spaces
$\X$ and $\X'$ such that $G\act\X$ is growth tight and $G\act \X'$ is
not.

In subsequent work \cite{CasTao14}, we extend the techniques of this paper to produce
the first examples of groups that admit a growth tight action on one
of their Cayley graphs and a non-growth tight action on another.
This answers in the affirmative the following question of Grigorchuk and de la Harpe \cite{GriDeL97}:
\begin{question}
  Does there exist a word metric for which $F_2\times F_2$ is growth
  tight?
\end{question}
Recall that $F_2\times F_2$ is not growth tight with respect to a generating set that
is a union of free generating sets of the two factors.

More generally, a product of infinite
groups acting on the $l^1$ product of their Cayley graphs is not
growth tight.
Such $l^1$ products and the Dal'bo, Otal, Peign\'e examples are the only known general
constructions of non-growth tight examples.
It would be interesting to have a condition to exclude growth
tightness.
One can not hope to bound the growth exponents of quotients away from
that of the group, as Shukhov \cite{Shu99} and Coulon \cite{Cou13} have given
examples of hyperbolic groups and sequences of quotients 
whose growth exponents limit to that of the group.
At present, growth tightness can only be excluded for a particular
action by exhibiting a quotient of the group by an infinite normal subgroup whose growth exponent is equal to that
of the group.

\subsection{The Hopf Property}\label{sec:hopf}
A group $G$ is \emph{Hopfian} if there is no proper quotient of $G$
isomorphic to $G$.

Let $\mathfrak{D}$ be a set of pseudo-metrics on $G$ that is \emph{quotient-closed}, in the sense that if $\Gamma$ is a normal subgroup of $G$ such that there exists an isomorphism $\phi\from G\to G/\Gamma$, then for every $d\in\mathfrak{D}$, the pseudo-metric on $G$ obtained by pulling back via $\phi$ the pseudo-metric on $G/\Gamma$ induced by $d$ is also in $\mathfrak{D}$.
For example, the set of word metrics on $G$ coming from finite
generating sets is quotient-closed.

Suppose further that $\mathfrak{D}$ contains a minimal growth pseudo-metric $d_0$, i.e., $\rate_{G,d_0}=\inf_{d\in\mathfrak{D}}\rate_{G,d}$, and that $G$ is growth tight with respect to $d_0$.

\begin{proposition}\label{prop:hopf}
Let $G$ be a finitely generated group with a bound on the cardinalities of
  its finite normal subgroups.
Suppose that there exists a quotient-closed set $\mathfrak{D}$ of pseudo-metrics
on $G$ that contains a growth tight, minimal growth
element $d_0$ as above.
Then $G$ is Hopfian.
\end{proposition}
The hypothesis on bounded cardinalities of finite normal subgroups holds for all groups of interest in this paper, see \fullref{thm:maximalnormal}.
\begin{proof}
Suppose that $\Gamma$ is a normal subgroup of $G$ such that $G\cong G/\Gamma$.
Let $d$ be the pseudo-metric on $G$ obtained from pulling back the pseudo-metric on $G/\Gamma$ induced by $d_0$. 
Since $\mathfrak{D}$ is quotient-closed, $d\in\mathfrak{D}$.
By minimality, $\rate_{G,d_0}\leqslant \rate_{G,d}$, but by growth tightness, $\rate_{G,d}\leqslant \rate_{G,d_0}$, with equality only if $\Gamma$ is finite.
Thus, the only normal subgroups $\Gamma$ for which we could have $G\cong G/\Gamma$ are finite.
However, if $G\cong G/\Gamma$ for some finite $\Gamma$ then $G$ has
arbitrarily large finite normal subgroups, contrary to hypothesis.
\end{proof}

Grigorchuk and de la Harpe \cite{GriDeL97} suggested this as a
possible approach to the question of whether a non-elementary Gromov
hyperbolic group is Hopfian, in the particular case that $\mathfrak{D}$
is the set of word metrics on $G$.
Arzhantseva and Lysenok \cite{ArzLys02} proved that every word metric on a non-elementary hyperbolic group is growth tight.
They conjectured that the growth exponent of such a group achieves its
infinum on some finite generating set and proved a step towards this conjecture~\cite{ArzLys06}.  
Sambusetti \cite{Sam02amalgam} gave an examples of a (non-hyperbolic) group for which the set of word metrics does not realize its infimal growth exponent.
In general it is difficult to determine whether a given group has a generating set that 
realizes the infimal growth exponent among word metrics.
Part of our motivation for studying growth tight actions is to open new possibilities for the set $\mathfrak{D}$ of pseudo-metrics considered above.

  Torsion free hyperbolic groups are Hopfian by a theorem of Sela
  \cite{Sel99}. 
Reinfeldt and Weidmann \cite{ReiWei10} have announced a generalization
of Sela's techniques to hyperbolic groups with torsion, and concluded
that all hyperbolic groups are Hopfian.

\subsection{The Rank Rigidity Conjecture}\label{sec:rankrigidity}
The Rank Rigidity Conjecture \cite{CapSag11,BalBuy08} asserts that if $\X$ is a locally
compact, irreducible, geodesically complete CAT(0) space, and $G$ is an infinite
discrete group acting properly and cocompactly on $\X$, then one of
the following holds:
\begin{enumerate}
\item $\X$ is a higher rank symmetric space.\label{rritemsym}
\item $\X$ is a Euclidean building of dimension at least 2.\label{rritembldg}
\item $G$ contains a rank 1 isometry.\label{rritemrk1}
\end{enumerate}
 
In case (\ref{rritemsym}), the Margulis Normal Subgroup Theorem
implies that $G$ is trivially growth tight, since it has no infinite,
infinite index normal subgroups.
Conjecturally, the Margulis Normal Subgroup Theorem also holds in case (\ref{rritembldg}).
Our \fullref{cat0} says that if $\X$ is proper then $G\act \X$ is a growth tight action in case (\ref{rritemrk1}).
Thus, a non-growth tight action of a non-elementary group on a proper, 
irreducible CAT(0) space as above would provide a counterexample
either to the Rank Rigidity Conjecture or to the conjecture that the
Margulis Normal Subgroup Theorem applies to Euclidean buildings.

The Rank Rigidity Conjecture is known to be true for many interesting
classes of spaces, such as Hadamard manifolds \cite{Bal95},
2--dimensional, piecewise-Euclidean cell complexs \cite{BalBri95},
Davis complexes of Coxeter groups \cite{CapFuj10}, universal covers of
Salvetti complexes
of Right Angled Artin Groups
\cite{BehCha12}, and finite dimensional CAT(0) cube complexes
\cite{CapSag11}, so \fullref{cat0} provides many new examples of
growth tight actions.

It is unclear when growth tightness holds if $\X$ is reducible. 
A direct product of infinite groups acting via a product action on a
product space with the $l^1$ metric fails to be growth tight.
However, there are also examples \cite{BurMoz97} of infinite simple groups acting
cocompactly on products of trees.
In \cite{CasTao14} we find partial results in the case that the group
action is a product action.

\subsection{Outline of the Proof of the Main Theorems}
Sambusetti \cite{Sam02amalgam} proved that a non-elementary free
product of non-trivial groups has a greater growth exponent than that
of either factor.
Thus, a strategy to prove growth tightness is to find a subset of 
$G$ that looks like a free product, with one factor that grows
like the quotient group we are interested in. Specifically:
\begin{enumerate}
\item Find a subset $A\subset G\subset \X$ such that
  $\rate_{A}=\rate_{G/\Gamma}$.
We will obtain $A$ as a coarsely dense subset of a minimal
section of the quotient map
$G\to G/\Gamma$, see \fullref{def:A}.
\item Construct an embedding of a free product set
  $A*\mathbb{Z}_2$ into $\X$.
The existence of a strongly contracting element $h\in\Gamma$ is used in
the construction of this embedding, see \fullref{proposition:embedded}.\label{item:embedding}
\item Show that
  $\rate_{G/\Gamma}=\rate_{A,d_\X}<\rate_{A*\mathbb{Z}/2\mathbb{Z},d_\X}\leqslant\rate_G$.
In this step it is crucial that $A$ is divergent, see
\fullref{def:divergent} and \fullref{lemma:strictgrowthinequality}.
We use quasi-convexity/complementary growth exponent to establish divergence.
\end{enumerate}

This outline, due to Sambusetti, is nowadays standard.
Typically step~(\ref{item:embedding}) is accomplished by a Ping-Pong
argument, making use of fine control on the geometry of the space
$\X$.
Our methods are coarser than such a  standard approach, and therefore can be applied to a wider variety of spaces.
We use, in particular, a technique of Bestvina, Bromberg, and Fujiwara \cite{BesBroFuj10}
to construct an
action of $G$ on a quasi-tree.
Verifying that the map from the free product set into $\X$ is an
embedding amounts to showing that elements in $A$ do not cross certain
coarse edges of the quasi-tree.

\part{Growth Tight Actions}\label{part:gt}
\section{Preliminaries}
Fix a $G$--space $\X$.
From now on, $d$ is used to denote the metric on $\X$ as well as the
induced pseudo-metric on $G$ and $G/\Gamma$.
Since there will be no possibility of confusion, we suppress $d$ from
the growth exponent notation.

We denote by $\ball{r}{x}$ the open ball of radius $r$ about the point $x$ and by 
$\nbhd{r}{\A}:=\cup_{x\in\A}\ball{r}{x}$ the open $r$--neighborhood about the
set $\A$.
The closed $r$--ball and closed $r$--neighborhood are denoted
$\clball{r}{x}$ and $\clnbhd{r}{\A}$, respectively.

\subsection{Coarse Language}
All of the following definitions may be written without specifying $\consta$ to
indicate that some such $\consta\geqslant 0$ exists:
Two subsets $\A$ and $\A'$ of $\X$ are $\consta$--\emph{coarsely equivalent}
if $\A\subset\clnbhd{\consta}{\A'}$ and $\A'\subset\clnbhd{\consta}{\A}$.
A subset $\A$ of $\X$ is $\consta$--\emph{coarsely dense} if it is
$\consta$--coarsely equivalent to $\X$.
A subset $\A$ of $\X$ is $\consta$--\emph{coarsely connected} if for every
$a$ and $a'$ in $\A$ there exists a chain $a=a_0,\, a_1,\dots,\, a_n=a'$
of points in $\A$ with $d(a_i,a_{i+1})\leqslant \consta$.
 
A \emph{pseudo-map} $\phi\from\X\to\Y$ assigns to each point in $\X$ a
subset $\phi(x)$ of $\Y$.
A pseudo-map is $\consta$--\emph{coarsely well defined} if for every
$x\in\X$ the set $\phi(x)$ of $\Y$ has diameter at most $\consta$.
Pseudo-maps $\phi$ and $\phi'$ with the same domain and codomain are
$\consta$--\emph{coarsely equivalent} or $\consta$--\emph{coarsely agree} if $\phi(x)$ is $\consta$--coarsely
equivalent to $\phi'(x)$ for every $x$ in the domain.
A $\consta$--coarsely well defined pseudo-map is called a $\consta$--\emph{coarse map}. 
From a $\consta$--coarse map we can obtain a
$\consta$--coarsely equivalent map by selecting one point from every image
set. 
Conversely:
\begin{lemma}\label{lemma:equivariantcoarsemap}
  If $\phi\from \X\to\Y$ is coarsely $G$--equivariant then there is an
  equivariant coarse map coarsely equivalent to $\phi$.
\end{lemma}
\begin{proof}
Suppose there is a $\consta$ such that
$d(g.\phi(x),\phi(g.x))\leqslant\consta$ for all $x\in\X$ and $g\in G$.
Define $\phi'(x):=\bigcup_{g\in G}\inv{g}\!.\phi(g.x)$. Then $\phi'$ is
$G$--equivariant and $C$--coarsely equivalent to $\phi$.
\end{proof}

\begin{definition}
  If $\phi\from\X\to\Y$ is a pseudo-map and $\A$ and $\A'$ are subsets
  of $\X$, let $d^\phi(\A,\A')$ denote the diameter of $\phi(\A)\cup\phi(\A')$.
\end{definition}

\begin{definition}\label{def:quasiconvex}
  A subset $\A\subset\X$ is $\consta$--\emph{quasi-convex} if for every $a_0,\,a_1\in \A$ there exists a geodesic
  $\gamma$ between $a_0$ and $a_1$ such that $\gamma\subset
  \clnbhd{\consta}{\A}$.
It is $\consta$--\emph{strongly quasi-convex} if every geodesic with
endpoints in $\A$ stays in $\clnbhd{\consta}{\A}$.
\end{definition}

\begin{definition}\label{def:qcgspace}
  A $G$--space $\X$ is $\consta$--\emph{quasi-convex} if it contains a $\consta$--quasi-convex $G$--orbit.
\end{definition}
For convenience, if $\X$ is a quasi-convex $G$--space we assume
we have chosen a basepoint $\bp\in\X$ such that $G.\bp$ is quasi-convex.

A group is \emph{elementary} if it has a finite index cyclic subgroup.

\begin{definition}
  Let $g\in G$. The \emph{elementary closure} of $g$, denoted by $E(g)$, is the
  largest virtually cyclic subgroup containing $g$, if such a subgroup exists.
\end{definition}

A map $\phi\from\X\to\Y$ is an
$(\constm,\consta)$--\emph{quasi-isometric embedding}, for some
$\constm\geqslant 1$ and $\consta\geqslant 0$, if, for all $x_0,\,x_1\in\X$:
\[\frac{1}{\constm}d(x_0,x_1)-\consta\leqslant
d(\phi(x_0),(x_1))\leqslant \constm d(x_0,x_1)+\consta\]
A map $\phi$ is \emph{$\consta$--coarsely $\constm$--Lipschitz} if the second
inequality holds, and is a \emph{quasi-isometry} if it is a
quasi-isometric embedding whose image is $\consta$--coarsely dense.

An \emph{$(\constm,\consta)$--quasi-geodesic} is an $(\constm,\consta)$--quasi-isometric embedding
of a coarsely connected subset of $\mathbb{R}$.
If $\gamma\from I\to\X$ is a quasi-geodesic we let $\gamma_t$ denote
the point $\gamma(t)$, and let $\gamma$ denote the image of $\gamma$
in $\X$.

\begin{definition}
  A quasi-geodesic $\qgeo$ is \emph{Morse} if for every $\constm\geqslant 1$
  there exists a $\constb\geqslant 0$ such that every
  $(\constm,\constm)$--quasi-geodesic with endpoints on
  $\qgeo$ is contained in the $\constb$--neighborhood of $\qgeo$. 
\end{definition}

We will use notation to simplify some calculations. 
Let $\consta$ be a `universal constant'.
For us this will usually mean a constant that depends on $G\act \X$
and a choice of $\bp\in\X$, but not on the point in $\X$ at which quantities $a$
and $b$ are
calculated. 
\begin{itemize}
\item For $a\leqslant \consta b$ we write $a\lmul b$.
\item For $\frac{1}{\consta}b\leqslant a\leqslant \consta b$ we write $a\emul b$.
\item For $a\leqslant b+\consta$ we write $a\ladd b$.
\item For $b-\consta\leqslant a\leqslant b+\consta$ we write $a\eadd b$.
\item For $a\leqslant \consta b+\consta$ we write $a\laddmul b$.
\item For $\frac{1}{\consta}b-\consta\leqslant a\leqslant \consta b+\consta$ we write $a\eaddmul b$.
\end{itemize}

\subsection{Poincar\'e Series and Growth}\label{sec:poincare}
Let $(\X,\bp,d)$ be a pseudo-metric space with choice of basepoint.
Let $|x|:=d(\bp,x)$ be the induced semi-norm.
  Define the \emph{Poincar\'e series} of  $\A\subset \X$ to be
\[\Theta_{\A}(s):=\sum_{a\in \A}\exp(-s|a|)\]

Another related series is:
\[\Theta'_{\A}(s):=\sum_{n=0}^\infty \#(\clball{n}{\bp}\cap\A)\cdot\exp(-sn)\]

The series $\Theta_\A$ and $\Theta_\A'$ have the same convergence
behavior, since
$\Theta_{\A}(s)=\Theta'_{\A}(s)\cdot (1-\exp(-s))$.
It follows that the growth exponent of $\A$ is a \emph{critical
  exponent} for $\Theta_\A'$ and $\Theta_\A$: the series converge for $s$
greater than the critical exponent and diverge for $s$ less than the
critical exponent.

\begin{definition}\label{def:divergent}
  $\A\subset \X$ is \emph{divergent} if $\Theta_\A$ diverges at its
  critical exponent.
\end{definition}

Since point stabilizers are finite, if $A<G$ and we set $\A:=A.\bp$
then $\Theta_A\emul \Theta_\A$ and $\Theta'_A\emul\Theta_\A'$.
This implies $\rate_A=\rate_\A$, so we can compute the growth exponent
of $A$ with respect to the pseudo-metric on $A$ induced by $G\act\X$ by
computing the growth exponent of the $A$--orbit as a subset of $\X$.

\subsection{The Quasi-tree Construction}
We recall the method of Bestvina, Bromberg, and Fujiwara
\cite{BesBroFuj10} for producing group actions on quasi-trees.
A \emph{quasi-tree} is a geodesic metric space that is quasi-isometric
to a simplicial tree. 
Manning \cite{Man05} gave a characterization of quasi-trees as spaces
satisfying a `bottleneck' property.
We use an equivalent formulation:
 
\begin{definition}[(Bottleneck Property)]
  A geodesic metric space satisfies the \emph{bottleneck property} if
  there exists a number $\bottleneck$ such that for all $x$ and $y$ in $\X$, and
  for any point $m$ on a geodesic segment from
  $x$ to $y$, every path from $x$ to $y$ passes through 
  $\clball{\bottleneck}{m}$.
\end{definition}

\begin{theorem}[{\cite[Theorem 4.6]{Man05}}]
A geodesic metric space is a quasi-tree if and only if it satisfies
the bottleneck property.
\end{theorem}

Let $\mathbb{Y}$ be a collection of geodesic metric spaces, and
suppose for each $X,\, Y\in\mathbb{Y}$ we have a subset
$\pi_Y(X)\subset Y$, which is referred to as the \emph{projection of
  $X$ to $Y$}.
Let $d_Y^\pi(X,Z):=\diam \pi_Y(X)\cup\pi_Y(Z)$.

\begin{definition}[(Projection Axioms)]\label{def:projectionaxioms}
A set $\mathbb{Y}$ with projections as above satisfies the
\emph{projection axioms} if there exist $\xi\geqslant 0$ such that for all distinct $X,Y,Z\in\mathbb{Y}$:
  \begin{enumerate}
\item[(P0)]  $\diam\pi_Y(X)\leqslant \xi$
\item[(P1)] At most one of $d_X^\pi(Y,Z)$, $d_Y^\pi(X,Z)$, or
  $d_Z^\pi(X,Y)$ is strictly greater than $\xi$.
\item[(P2)] $|\{V\in\mathbb{Y}\mid
  d^\pi_V(X,Y)>\xi\}|<\infty$
  \end{enumerate}
\end{definition}

For a motivating example, let $G$ be the fundamental group of a
closed hyperbolic surface, and let $\H$ be the axis in $\mathbb{H}^2$
of $h\in G$.
Let $\mathbb{Y}$ be the distinct $G$--translates of $\H$, and for each
$Y\in\mathbb{Y}$ let $\pi_Y$ be closest point projection to $Y$. 
In this example, projection distances arise as closest point
projection in an ambient space containing $\mathbb{Y}$. 
Bestvina, Bromberg, and Fujiwara consider abstractly the collection $\mathbb{Y}$
and projections satisfying the projection axioms, and build
an ambient space containing a copy of $\mathbb{Y}$ such that closest
point projection agrees with the given projections, up to bounded
error:

\begin{theorem}[{\cite[Theorem~A and Theorem~B]{BesBroFuj10}}]\label{thm:BBF}
  Consider a set $\mathbb{Y}$ of geodesic metric spaces and projections satisfying the projection axioms.
There exists a geodesic metric space $\Y$ containing disjoint,
isometrically embedded, totally geodesic  copies of each
$Y\in\mathbb{Y}$, such that for $X,\,Y\in\mathbb{Y}$, closest
point projection of $X$ to $Y$ in $\qt$ is uniformly coarsely
equivalent to $\pi_Y(X)$.

The construction is equivariant with respect to any group action that
preserves the projections.
Also, if each $Y\in\mathbb{Y}$ is a quasi-tree, with uniform
bottleneck constants, then $\qt$ is a quasi-tree.
\end{theorem}

The basic idea is that $Z$ is `between' $X$ and $Y$ in $\mathbb{Y}$ if
$d^\pi_Z(X,Y)$ is large, and $X$ and $Y$ are `close' if there is no
$Z$ between them.
Essentially, $\qt$ is constructed by choosing parameters $\consta$ and $\constb$
and connecting every point of
$\pi_Y(X)$ to every point of $\pi_X(Y)$ by an edge of length $\constb$ if
there does not exist $Z\in\mathbb{Y}$ with $d^\pi_Z(X,Y)>\consta$.
For technical reasons one actually must perturb the projection
distances a bounded amount first.
Then, if $\consta$ is chosen sufficiently large and $\constb$ is chosen sufficiently
large with respect to $\consta$, the resulting space is the $\Y$ of \fullref{thm:BBF}.

\subsection{Hyperbolically Embedded Subgroups}
Dahmani, Guirardel, and Osin \cite{DahGuiOsi11} define the concept of a
\emph{hyperbolically embedded subgroup}.
This is a generalization of a peripheral subgroup of a relatively hyperbolic group.
We will not state the definition, as it is technical and we will not
work with this property directly, but it follows from \cite[Theorem~4.42]{DahGuiOsi11} that $E(h)$ is
hyperbolically embedded in $G$ for any strongly contracting element
$h$.
The proof of this theorem proceeds by considering the action of $E(h)$
on a quasi-tree constructed via the method of Bestvina, Bromberg, and Fujiwara.

We state some results on
hyperbolically embedded subgroups that are
related to the work in this paper.
These are not used in the proofs of the main theorems.

\begin{theorem}[{\cite[Theorem~2.23]{DahGuiOsi11}}]\label{thm:maximalnormal}
  If $G$ has a hyperbolically embedded subgroup then $G$ has a maximal finite normal subgroup.
\end{theorem}
Recall that this theorem guarantees one of the hypotheses of \fullref{prop:hopf}.

\begin{theorem}\label{manyquotients}
  If $G$ contains an infinite order element $h$ such that $E(h)$ is
  hyperbolically embedded then $G$ has an
  infinite, infinite index normal subgroup.
\end{theorem}
\begin{proof}
By \cite[Theorem 5.15]{DahGuiOsi11}, for a sufficiently large $n$, the
normal closure of $\langle h^n\rangle$ in $G$ is the free product of the
conjugates of $\langle h^n\rangle$.
\end{proof}

This theorem says that our main results are true for interesting
reasons, not simply for
lack of normal subgroups.

 Minasyan and Osin \cite{MinOsi13} produce hyperbolically embedded
 subgroups in certain graphs
 of groups.
We use these to produce growth tight examples in \fullref{cat0gog}.
\begin{theorem}[{\cite[Theorem~4.17]{MinOsi13}}]\label{corollary:gog}
  Let $G$ be a finitely generated, non-elementary group that splits non-trivially as a graph of
  groups and is not an ascending HNN-extension.
If there exist two edges of the corresponding
  Bass-Serre tree whose stabilizers have finite intersection then $G$ contains an infinite order element $h$ such that $E(h)$ is hyperbolically embedded in $G$.
\end{theorem}

\section{Contraction and Constriction}
In this section we introduce properties called
`contracting' and `constricting' that generalize properties of closest
point projection to a geodesic in hyperbolic space, and verify that
the `strong' versions of these properties are sufficient to satisfy
the projection axioms of \fullref{def:projectionaxioms}.
These facts are well known to the experts\footnote{For example,
  \cite{Sis11} shows the projection axioms are satisfied for
  constricting elements, without assuming that $\X$ is proper.}, but as
there is currently no published general treatment of this material, we provide a
detailed account.

\subsection{Contracting and Constricting}
In this section we define contracting and constricting maps and show
that the strong versions of these properties are equivalent.

\begin{definition}
A $\consta$--coarse map  $\pi\from\X\to\A$ is \emph{$\consta$--\ccpp} if for
  all $x$ there exists an $a\in \A$ with $d(x,\A)=d(x,a)$ such that
  $\diam \{a\}\cup \pi(x)\leqslant \consta$.
\end{definition}
Recall $d^\pi(x_0,x_1):=\diam \pi(x_0)\cup\pi(x_1)$.
\begin{definition}\label{def:contracting}
  $\pi\from \X\to \A$ is $(\constm,\consta)$--\emph{contracting} for
  $\consta\geqslant 0$ and $\constm\geqslant 1$ if
  \begin{enumerate}
\item $\pi$ and $\mathrm{Id}_\A$ are $\consta$--coarsely equivalent on
  $\A$, and
\item $d(x_0,x_1)<\frac{1}{\constm}d(x_0,\A)-\consta$ implies
  $d^\pi(x_0,x_1)\leqslant \consta$ for
  all $x_0,x_1\in\X$.\label{item:contractingdefballcontracting}
  \end{enumerate}
We say $\pi$ is
\emph{strongly contracting} if it is $(1,
\consta)$--contracting and $d(x,\pi(x))- d(x,\A)\leqslant \consta$ for
all $x\in\X$.
\end{definition}

Another formulation of strong contraction says that geodesics far from
$\A$ have
bounded projections to $\A$:

\begin{definition}\label{def:boundedgeodesicimage}
 A coarse map $\pi\from \X\to\A$ has the \emph{Bounded Geodesic Image Property} if
  there is a constant $\consta$ such that for every geodesic
  $\geo$, if $\geo\cap \nbhd{\consta}{\A}=\emptyset$ then
  $\diam(\pi(\geo))\leqslant \consta$.
\end{definition}

\begin{lemma}\label{lemma:boundedgeodesicimage}
If  $d(x,\pi(x))- d(x,\A)$ is uniformly bounded then $\pi$ has the Bounded Geodesic
Image Property if and only if it is strongly contracting.
\end{lemma}
\begin{proof}
First, assume that $\pi$ has the Bounded Geodesic Image Property, for
some constant $\consta$.
Let $x$ be any point in $\X\setminus \nbhd{\consta}{\A}$.
For any $y$ such that $d(x,y)<d(x,\A) - \consta$, every geodesic
from $x$ to $y$ remains outside $\nbhd{\consta}{\A}$, so its projection has diameter
at most $\consta$.

For the converse,
  suppose $\pi\from \X\to\A$ is a $\consta$--coarse map that is $(1,\consta)$--contracting and
 $d(x,\pi(x))- d(x,\A)\leqslant \consta$ for all $x\in\X$.
If $\consta=0$ then balls outside of $\nbhd{\consta}{\A}$ project to a
single point, and we are done, so assume $\consta> 0$.
Let $\geo\from [0,T]\to\X$ be a geodesic that stays outside $\nbhd{3\consta}{\A}$.  
Let $t_0:=d(\geo_0,\A)-\consta$, and let $s:=T-d(\geo_T,\A)+\consta$. 
If $s\leqslant t_0$ then $d^\pi(\geo_0,\geo_T)\leqslant 2\consta$.
Otherwise, define $t_{i+1}:=t_i+d(\geo_{t_i},\A)-\consta$, provided
$t_{i+1}<s$.
Each $t_{i+1}-t_i\geqslant 2\consta$, so we have
 a partition of $[0,T]$ into subintervals
$[0,t_0],[t_0,t_1],\dots,[t_{k-1},t_k],[t_k,s],[s,T]$ with
$k<\frac{s-t_0}{2\consta}$, and if $[a,b]$ is one of these intervals then
$d^\pi(\geo_a,\geo_b)\leqslant \consta$, by strong contraction.

Now,
\begin{align*}
  d(\geo_0,\geo_T)&\leqslant
  d(\geo_0,\pi(\geo_0))+d(\pi(\geo_0),\pi(\geo_{t_0}))+d(\pi(\geo_{t_0}),\pi(\geo_s))\\
&\qquad\qquad+d(\pi(\geo_s),\pi(\geo_T))+d(\pi(\geo_T),\geo_T)\\
&\leqslant
d(\geo_0,\pi(\geo_0))+d(\pi(\geo_T),\geo_T)+\consta(3+\frac{s-t_0}{2\consta}),\\
\end{align*}
and
\begin{align*}
  d(\geo_0,\geo_T)&=d(\geo_0,\geo_{t_0})+d(\geo_{t_0},\geo_s)+d(\geo_s,\geo_T)\\
&=d(\geo_0,\A)-\consta + s-t_0 +d(\geo_T,\A)-\consta,
\end{align*}
so \[s-t_0\leqslant 2(5\consta+d(\geo_0,\pi(\geo_0))-d(\geo_0,\A)+
d(\geo_T,\pi(\geo_T))-d(\geo_T,\A))\leqslant 14\consta. \]
This means $k< 7$, so
$d^\pi(\geo_0,\geo_T)\leqslant \consta(3+k)< 10\consta$.
\end{proof}

  If $\pi$ is only $(\constm,\consta)$--contracting then a similar argument shows that $d^\pi(\geo_0,\geo_T)$ is
  bounded in terms of $\consta$ and $\log_{\frac{\constm+1}{\constm-1}}(d(\geo_0,\A)d(\geo_T,\A))$.

We now introduce the notion of a constricting map. 
Using constricting maps will simplify some of our proofs, but it turns
out that the strong versions of constricting and contracting are equivalent. 

\begin{definition}\label{def:pathsystem}
  A \emph{path system} is a transitive collection of
  quasi-geodesics with uniform constants that is closed under taking
  subpaths.

 A path system is \emph{minimizing} if, for some $\consta\geqslant 0$, it contains a path system
  consisting of $(1,\consta)$--quasi-geodesics.
\end{definition}

\begin{definition}\label{def:constricting}
  Let $\ps$ be a path system.
A coarse map $\pi\from \X\to \A$ is
\emph{$(\constm,\consta)$--$\ps$--constricting}\footnote{Sisto \cite{Sis11} calls
  this property `$\ps$--contracting'. We change the name to avoid
  conflict with the better established `contracting' terminology of
  \fullref{def:contracting}.} for $\constm\geqslant 1$ and
$\consta\geqslant 0$ if:
\begin{enumerate}
\item $\ps$ contains a path system consisting of $(\constm,\consta)$--quasi-geodesics,
\item $\pi$ and $\mathrm{Id}_\A$ are $\consta$--coarsely equivalent on
  $\A$, and
\item for every $\path\in\ps$ with endpoints $x_0$ and $x_1$, if
  $d^\pi(x_0,x_1)>\consta$ then $d(\pi(x_i),\path)\leqslant \consta$
  for both $i\in\{0,1\}$.
\end{enumerate}
A coarse map is \emph{constricting} if it is $(\constm,\consta)$--$\ps$--constricting for
some path system $\ps$ and \emph{strongly constricting} if it is $(1,\consta)$--constricting for the
path system consisting of all geodesics.
\end{definition}

\begin{lemma}\label{lemma:constrictingimpliescontracting}
  If $\pi\from\X\to\A$ is constricting then it is contracting.
\end{lemma}
\begin{proof}
Suppose $\pi$ is $(\constm,\consta)$--$\ps$-constricting
$\consta$--coarse map for a path system $\ps$
consisting of $(\constm,\consta)$--quasi-geodesics.
Suppose $\path\from [0,T]\to\X$ is a path in $\ps$ with $\path_0=x$
and $\path_T=y$, and suppose $z=\path_s\in\clnbhd{\consta}{\A}$.
Using the fact that $\path$ is an $(\constm,\consta)$--quasi-geodesic
on the intervals $[0,T]$, $[0,s]$, and $[s,T]$, one sees that $d(x,y)\geqslant\frac{1}{\constm^2}(d(x,\A)+d(y,\A)-4\consta)$.
Therefore, if $d(x,y)<\frac{1}{\constm^2}d(x,\A)-\frac{4\consta}{\constm^2}$
then $\path$ can not enter $\clnbhd{\consta}{\A}$. This would
contradict the constricting property, unless $d^\pi(x,y)\leqslant \consta$.
Therefore $\pi$ is $(\constm^2,\max\{\consta,\frac{4\consta}{\constm^2}\})$--contracting.
\end{proof}

\begin{lemma}\label{lemma:cpp}
 Let $\pi\from\X\to\A$ be an $\consta$--coarse map that is $(1,\consta)$--$\ps$--constricting.
For all $x\in\X$ and all $r\geqslant 0$ we have
$\{a\in\A\mid d(x,a)\leqslant d(x,\A)+r\}\subset\{a\in\A\mid
d(a,\pi(x))\leqslant r+5\consta\}$.

  In particular, setting $r=0$ shows that closest point
  projection to $\A$ is coarsely well defined and coarsely equivalent
  to $\pi$.
\end{lemma}
\begin{proof}
For $x\in\X$ and $r\geqslant 0$, let $a\in\A$ be a point such that $d(x,a)\leqslant d(x,\A)+r$.
Let $\path$ be a $(1,\consta)$--quasi-geodesic from $x$ to $a$ in $\ps$.
If $d(a,\pi(x))>2\consta$ then $d^\pi(a,x)>\consta$, so there is a point
$z\in\path\cap\clnbhd{\consta}{\pi(x)}$.
Now $d(x,z)+\consta\geqslant d(x,\pi(x))\geqslant d(x,\A)\geqslant d(x,a)-r$.
Since $\path$ is a $(1,\consta)$--quasi-geodesic, $d(x,a)\geqslant
d(x,z)+d(z,a)-3\consta$, so
$d(z,a)\leqslant r+4\consta$, and $d(a,\pi(x))\leqslant r+5\consta$.
\end{proof}

\begin{proposition}\label{prop:allthesame}
  Let $\pi\from\X\to\A$. The following are
  equivalent:
  \begin{enumerate}
\item $\pi$ is strongly constricting.\label{item:constrictallgeodesics}
  \item $\pi$ is constricting for some minimizing path system.\label{item:constrictsomegeodesics}
\item $\pi$ is strongly contracting.\label{item:strcontraction}
\item $\pi$ has the Bounded Geodesic Image Property and $d(x,\pi(x))-
  d(x,\A)$ is uniformly bounded.\label{item:thebgi}
  \end{enumerate}
\end{proposition}
\begin{proof}
(\ref{item:constrictallgeodesics}) implies (\ref{item:constrictsomegeodesics})
is immediate.

  Suppose $\pi$ is $(1,\consta)$--$\ps$--constricting for a minimizing path
  system $\ps$ consisting of $(1,\consta)$--quasi-geodesics.
\fullref{lemma:constrictingimpliescontracting} shows $\pi$ is $(1,\consta')$--contracting.
By \fullref{lemma:cpp}, $\pi$ is coarsely a closest point projection, so $d(x,\pi(x))-
d(x,\A)$ is uniformly bounded.
Thus, (\ref{item:constrictsomegeodesics}) implies (\ref{item:strcontraction}).

Now suppose $\pi$ is $(1,\consta)$--contracting and $d(x,\pi(x))-
d(x,\A)\leqslant\consta$ for all $x\in\X$.
Take any geodesic $\geo\from [0,T]\to\X$.
If $d^\pi(\geo_0,\geo_T)>10\consta$ then
$\geo\cap\nbhd{3\consta}{\A}\neq\emptyset$, as in \fullref{lemma:boundedgeodesicimage}.
Let $t=t_0,\,t_1$ be the first and last times, respectively,
such that $d(\geo_t,\A)\leqslant 3\consta$.
By \fullref{lemma:boundedgeodesicimage}, $d^\pi(\geo_0,\geo_{t_0})\leqslant 10\consta$.
Thus, $d(\pi(\geo_0),\geo_{t_0})\leqslant
d^\pi(\geo_0,\geo_{t_0})+d(\pi(\geo_{t_0}),\geo_{t_0})\leqslant 14\consta$.
The same argument shows $d(\pi(\geo_T),\geo_{t_1})\leqslant 14\consta$, so
$\pi$ is $(1,14\consta)$--constricting for the path system of all geodesics.
Thus, (\ref{item:strcontraction}) implies
(\ref{item:constrictallgeodesics}).

(\ref{item:strcontraction}) is equivalent to (\ref{item:thebgi}) by \fullref{lemma:boundedgeodesicimage}.
\end{proof}

\subsection{Additional Properties of Contracting and Constricting Maps}
We establish some properties of contracting and constricting maps that
will be useful in the sequel.

\begin{lemma}\label{lemma:distancelb}
  If $\pi$ is a $(1,\consta)$--strongly constricting $C$--coarse map
  and $d^\pi(x,y)>\consta$ then $d(x,y)\geqslant d(x,\pi(x))+d^\pi(x,y)+d(\pi(y),y)-6C$.
\end{lemma}
\begin{proof}
Let $\geo$ be a geodesic from $x$ to $y$.
By strong constriction, there exist $s$ and $t$ such that
$d(\geo_s,\pi(x))\leqslant\consta$ and $d(\geo_t,\pi(y))\leqslant\consta$.
The lemma follows from the triangle inequality and the fact that
$\pi(x)$ and $\pi(y)$ have diameter at most $\consta$.
\end{proof}

\begin{lemma}\label{lemma:coarse1Lipschitz}
  If $\pi\from\X\to\A$ is strongly constricting then it is coarsely
  1--Lipschitz.
\end{lemma}
\begin{proof}
  Let $\pi$ be an $\consta$--coarse
  map that is $(1,\consta)$--constricting on the path system
  of geodesics.
Let $x_0$ and $x_1$ be arbitrary points, and let $\geo$ be a
geodesic from $x_0$ to $x_1$.
If $d^\pi(x_0,x_1)>4\consta$ then
$\geo\cap\ball{\consta}{x_i}\neq\emptyset$ for each $i$, which implies
$d(x_0,x_1)\geqslant
d(x_0,\pi(x_0))+d^\pi(x_0,x_1)+d(\pi(x_1),x_1)-8\consta$.
Thus, for all $x_0$ and $x_1$, we have $d^\pi(x_0,x_1)\leqslant d(x_0,x_1)+8\consta$.
\end{proof}

\begin{lemma}\label{lemma:extendedMorse}
  Let $\pi\from\X\to\A$ be an $(\constm,\consta)$--contracting
  $\consta$--coarse map such that $d(x,\pi(x))- d(x,\A)\leqslant\consta$ for
  all $x\in\X$.
Fix $\constb\geqslant 1$.
For all sufficiently large $\constc$ there exists a $T_{max}$ such that if
$\qgeo\from [0,T]\to\X$ is a $(\constb,\constb)$--quasi-geodesic with
$d(\qgeo_0,A)=\constc=d(\qgeo_T,\A)$ and
$\qgeo\cap\nbhd{\constc}{\A}=\emptyset$ then $T\leqslant T_{max}$.
\end{lemma}
\begin{proof}
Let $\constc>\constm(\constb^2\consta+\consta+\constb)$.
Let $t_0:=0$ and let $t_{i+1}$ be the last time that
$d(\qgeo_{t_i},\qgeo_{t_{t+1}})=\frac{1}{\constm}d(\qgeo_{t_i},\A)-\consta$, provided
$t_{i+1}<T$.
This subdivides $[0,T]$ into at most
$1+\frac{T\constb}{\frac{\constc}{\constm}-\consta-\constb}$ intervals
$[t_0,t_1],\dots,[t_k,T]$, each of which has endpoints whose
$\pi$--images are distance at most $\consta$ apart.

Since $\qgeo$ is a quasi-geodesic,
$T\leqslant \constb d(\qgeo_0,\qgeo_T)+\constb^2$.
On the other hand:
\[d(\qgeo_0,\qgeo_T)\leqslant
2\constc+2\consta+d^\pi(\qgeo_0,\qgeo_T)\leqslant  2\constc+2\consta+\consta\left(
1+\frac{T\constb}{\frac{\constc}{\constm}-\consta-\constb}\right)\]
Combined with the
condition on $\constc$\!, this yields an upper bound on $T$.
\end{proof}

\begin{corollary}\label{corollary:Morse}
If $\pi\from\X\to\A$ is contracting and $d(x,\pi(x))-d(x,\A)$ is
uniformly bounded, then for all $\constm\geqslant
1$ and $\constc\geqslant 0$ there exists a $\constb$ such that every
$(\constm,\constm)$--quasi-geodesic with endpoints at distance at
most $\constc$ from $\A$ is contained in $\clnbhd{\constb}{\A}$.

In particular, if $\A$ is a quasi-geodesic then it is Morse.
\end{corollary}

\begin{lemma}\label{lemma:contractingquasigeodesicsarestable}
  Let $\qgeo\from\mathbb{R}\to\X$ be a
 quasi-geodesic, and let $\pi\from\X\to\qgeo$ be a
  strongly contracting projection. 
For all $\constc\geqslant 0$ there exists a $
\constb$ such that if
$\path\from [0,T]\to\X$ is
a geodesic and $t_0$ and $t_1$ are such that
$d(\path_0,\qgeo_{t_0})\leqslant \constc$ and
$d(\path_T,\qgeo_{t_1})\leqslant \constc$ then
$\qgeo_{[t_0,t_1]}\subset \clnbhd{\constb}{\path}$.
\end{lemma}
\begin{proof}
By \fullref{prop:allthesame}, $\pi$ is strongly constricting, so
$\path$ passes close to every point in $\pi(\path)$. 
Let $i$ and $j$ be numbers in the domain of $\path$, with $0<j-i\leqslant 1$.
Let $s_i$ and $s_j$ be such that $\qgeo_{s_i}\in\pi(\path_i)$ and $\qgeo_{s_j}\in\pi(\path_j)$.
Then $s_i$ and $s_j$ are 
 boundedly far apart, since
 $\pi$ is coarsely 1-Lipschitz, by
\fullref{lemma:coarse1Lipschitz}, and $\qgeo$ is a quasi-geodesic.
Therefore, the diameter of $\qgeo_{[s_i,s_j]}$ is bounded,
and we have already noted that $\qgeo(s_i)$ and $\qgeo(s_j)$ are
close to $\path$, since they are in the image of $\pi$.
\end{proof}

\begin{lemma}\label{lemma:strongconstrictiscoarseequivinvariant}
Let $\A$ and $\A'$ be coarsely equivalent subsets of
$\X$.
 Let $\sigma\from \A\to\A'$ and
 $\bar{\sigma}\from\A'\to\A$ be $\consta$--coarse maps such that
 $d(a,\sigma(a))\leqslant \consta$ for all $a\in\A$ and $d(a',\bar{\sigma}(a'))\leqslant
 \consta$ for all $a'\in\A'$. 
Then $\pi_\A\from\X\to\A$ is strongly contracting if and only if
$\pi_{\A'}:=\sigma\circ\pi_\A\from\X\to\A'$ is strongly contracting.
\end{lemma}
\begin{proof}
  Suppose $\pi_\A$ is $(1,\consta)$--contracting and $d(x,\pi(x))-
  d(x,\A)\leqslant\consta$ for all $x\in\X$.
If $d(x,y)\leqslant d(x,\A')-2\consta\leqslant d(x,\A)-\consta$ then
$d^\pi_{\A'}(x,y)\leqslant d^\pi_\A(x,y)+2\consta\leqslant 3\consta$, so
$\pi_{\A'}$ is $(1,3\consta)$--contracting.

Take a point $x$ and let $a'\in\A'$ such that
$d(x,\A')=d(x,a')$.
Then $d(x,\bar{\sigma}(a'))-\consta\leqslant d(x,a')\leqslant d(x,\pi_{\A'}(x))\leqslant
d(x,\pi_\A(x))+2\consta$, so $d(x,\bar{\sigma}(a'))\leqslant d(x,\A)+3\consta$.
By \fullref{prop:allthesame}, $\pi_\A$ is strongly constricting, so by
\fullref{lemma:cpp}, there is a constant $\constc$ such
that $d(\pi_\A(x),\bar{\sigma}(a'))\leqslant 3\consta+\constc$.
Thus, $\pi_{\A'}$ is
$(5\consta+\constc)$--\ccpp, hence, strongly contracting.
\end{proof}

\begin{lemma}\label{lemma:boundedimage}
  Let $\pi\from\X\to\A$ be strongly constricting.
There exists a number $\constb$ such that if $d(\A,g\A)>\constb$ then
$\diam\pi(g\A)$ is bounded, independent of $g$.
\end{lemma}
\begin{proof}
  Let $\pi$ be $(1,\consta)$--strongly constricting.
By \fullref{prop:allthesame}, $\pi$ is strongly contracting,
so by \fullref{corollary:Morse} there is a constant $\constb$ such that a
geodesic with endpoints in $\A$ stays in the $(\constb-\consta)$--neighborhood of
$\A$. 
Therefore, a geodesic with endpoints in $g\A$ stays in $\clnbhd{\constb-\consta}{g\A}$.
Choose $x\in g\A$ such that $d(x,\A)=d(g\A,\A)$.
For all $y\in g\A$, if $d^\pi(x,y)>\consta$ then a geodesic from $x$ to $y$ passes within distance $\consta$ of $\pi(x)$
and $\pi(y)$. 
This means 
$\clnbhd{\consta}{\A}\cap\clnbhd{\constb-\consta}{g\A}\neq\emptyset$, so $d(\A,g\A)\leqslant
\constb$.
Thus, if $d(\A,g\A)>\constb$, then $d^\pi(x,y)\leqslant \consta$, so $\diam
\pi(g\A)\leqslant 2\consta$.
\end{proof}

\subsection{Strongly Contracting
  Elements}\label{sec:stronglycontractingelements}
We have defined contraction and constriction for maps.
We now give definitions for group elements:

\begin{definition}\label{def:contractinglement}
  An element $h\in G$ is called \emph{contracting}, with
    respect to $G\act\X$,
  if $i \mapsto h^i\!.\bp$ is a quasi-geodesic and if there exists a
  subset $\A\subset\X$ on which $\langle h\rangle$ acts cocompactly
  and a map $\pi\from\X\to\A$ that is contracting.

An element $h\in G$ is called \emph{constricting}, with
    respect to $G\act\X$,
  if $i \mapsto h^i\!.\bp$ is a quasi-geodesic and if there exists a
  subset $\A\subset\X$ on which $\langle h\rangle$ acts cocompactly, a
  $G$--invariant path system $\ps$, 
  and a map $\pi\from\X\to\A$ that is $\ps$--constricting.

An element is \emph{strongly contracting} or \emph{strongly constricting} if the
projection $\pi$ is, respectively, strongly contracting or strongly constricting.
\end{definition}

For $\pi$ and $\A$ as in the definition, \fullref{prop:allthesame} says $\pi$ is strongly contracting if and
only if it is strongly constricting. 
Thus,
\fullref{lemma:cpp} says closest point projection to $\A$ is coarsely
well defined and coarsely equivalent to $\pi$.
\fullref{lemma:strongconstrictiscoarseequivinvariant} says that the choice
of the set $\A$ only affects the constants of strong contraction.
It follows that an element $h$ is strongly contracting if and only if
$i\mapsto h^i\!.\bp$ is a quasi-geodesic and closest point projection
to $\langle h\rangle.\bp$ is strongly contracting. 
In the remainder of this section we produce more finely
tailored choices for $\A$ and $\pi$. 
In particular, we would like $\pi$ to be compatible with the group
action, see \fullref{rem:equivariant}.

\begin{proposition}[{(cf.\ \cite[Lemma~6.5]{DahGuiOsi11})}]\label{corollary:Ehexists}
Let $G$ be a finitely generated group, and let $\X$ be a $G$--space.
Let $h\in G$ be an infinite order element.
If there exists a strongly constricting $\pi\from\X\to\langle
h\rangle.\bp$ then:
\[E(h)=H:=\{g\in
  G\mid g\langle h\rangle.\bp\text{ is coarsely equivalent to }\langle
  h\rangle.\bp\}\]
\end{proposition}
\begin{proof}
$H$ is a group containing every
  finite index supergroup of $\langle
  h\rangle$.
Let $\constc$ be the constant of \fullref{lemma:boundedimage}, and let
$S:=\{g\in\G\mid d(g\langle h\rangle.\bp,\langle h\rangle.\bp)\leqslant
\constc\}$. Then \fullref{lemma:boundedimage} implies $H\subset S$.
Since $G\act\X$ is properly discontinuous, $S$ is contained in
finitely many $h$--orbits, so $\langle h\rangle<H$ has  finite index.
Therefore, $E(h)$ exists and is equal to $H$.
\end{proof}

\begin{definition}\label{def:quasiaxis}
If $h$ is a strongly contracting element, define
 the \emph{(quasi)-axis} of $h$, with respect to the basepoint $\bp$,
 to be $\H:=E(h).\bp$.
 \end{definition}

\begin{lemma}\label{lemma:projection}
   If $h$ is a strongly contracting element then there exists an
  $E(h)$--equivariant, strongly contracting
  coarse map $\pi_\H\from\X\to\H$.
\end{lemma}
\begin{proof}
  By \fullref{prop:allthesame}, \fullref{lemma:cpp}, and \fullref{lemma:strongconstrictiscoarseequivinvariant}
  any choice of
  closest point projection map to $\H$ is strongly contracting and
  coarsely $E(h)$--equivariant, so, by
  \fullref{lemma:equivariantcoarsemap}, we can replace it by a
  coarsely equivalent, $E(h)$--equivariant coarse map, which will
  still be strongly contracting, by \fullref{lemma:strongconstrictiscoarseequivinvariant}.
\end{proof}

\begin{definition}\label{def:equivariantprojections}
From the projection $\pi_\H$ of \fullref{lemma:projection} define strongly contracting projections 
onto each translate of $\H$ by $\pi_{g\H}\from \X\to g\H: x\mapsto 
g.\pi_\H(g^{-1}\!.x)$. 
\end{definition}

If $g'\H=g\H$ then $g^{-1}g'\in E(h)$ so \fullref{lemma:projection} implies
$\pi_{g'\H}(x)=\pi_{g\H}(x)$ for all $x\in\X$.
\begin{remark}\label{rem:equivariant}
The projections of \fullref{def:equivariantprojections} satisfy $g.\pi_\H(x)=\pi_{g\H}(g.x)$ for all $x\in\X$ and 
$g\in G$.   
\end{remark}

\subsection{Strongly Contracting Elements and the Projection Axioms}\label{sec:strcontandproj}
Let $h\in G$ be a strongly contracting element with respect to $G\act
\X$.
Let $\H$ be a quasi-axis of $h$ defined in \fullref{def:quasiaxis}.
We wish to apply \fullref{thm:BBF} to the collection of
$G$--translates of $\H$ with the projections of \fullref{def:equivariantprojections}.
To see that the hypotheses of the theorem are satisfied, we first
embed $\H$ into a geodesic metric space and then verify the projection
axioms of \fullref{def:projectionaxioms}.

Choose  representatives
$1=g_0,\dots,g_{n-1}$ for
$\langle h\rangle\backslash E(h)$, so that for each $i$ we have $d(g_i.\bp,\bp)=\min_{g\in
  \langle h\rangle g_i}d(g.\bp,\bp)$.
Let $g_n:=h$.
Let $\hat{\H}$ be the Cayley graph of $E(h)$ with respect
to the generating set $\{g_1,\dots,g_n\}$.
The graph $\hat{\H}$ becomes a geodesic metric space by assigning each
edge length one, and it is a quasi-tree since $E(h)$ is virtually cyclic.

Choose representatives $1=f_0,f_1,\dots$ for $G/E(h)$.
Let $\mathbb{Y}$ be a disjoint union of copies of $\hat{\H}$, one for
each $f_iE(h)\in G/E(h)$, denoted $f_i\hat{\H}$.
The orbit map $f_i\hat{\H}\to f_i\H := f_ie\mapsto
f_ie.\bp$ is a quasi-isometric embedding, so its inverse
$\phi_{f_i\H}\from f_i\H\to f_i\hat{\H}$ is a coarse
map that is a quasi-isometry.
 Define $\pi_{f_i\hat{\H}}(f_j\hat{\H}):=\phi_{f_i}(\pi_{f_i\H}(f_j\H))$.
Since $\phi_{f_i}$ is a quasi-isometry it suffices
to check the projection axioms on translates of $\H$ in $\X$.

\begin{lemma}[Axiom (P0)]\label{lemma:axiom0}
There is a uniform bound on the diameter of $\pi_\H(g\H)$ for $g\notin E(h)$.
\end{lemma}
\begin{proof}Let $\pi_\H\from\X\to\H$ be $(1,\consta')$--strongly constricting.
Let $\qgeo\from \mathbb{R}\onto\H$ be an $(\constm,\consta'')$--quasi-geodesic
parameterization that agrees with $i\mapsto h^i\!.\bp$ on the integers. 
Replace $\consta'$ and $\consta''$ by $\consta:=\max\{\consta',\consta''\}$.

Let $\constc:=\diam\langle h\rangle\backslash\H$.
Let $\constb$ be large enough so that if $\path$ is a geodesic with
$d(\path_{s_0},\qgeo_{t_0})\leqslant \consta$ and
$d(\path_{s_1},\qgeo_{t_1})\leqslant \consta$ then
$\path_{[s_0,s_1]}\subset\clnbhd{\constb}{\qgeo_{[t_0,t_1]}}$ and
$\qgeo_{[s_0,s_1]}\subset\clnbhd{\constb}{\path_{[t_0,t_1]}}$, as in
\fullref{corollary:Morse} and \fullref{lemma:contractingquasigeodesicsarestable}.

Suppose $g\notin E(h)$.
For a pair of points $x_0,\,x_1\in g\H$, take $t_0$ and $t_1$ such that $\qgeo_{t_i}\in \pi_\H(x_i)$
for each $i$. 
Let $\path$ be a geodesic connecting $x_0$ to $x_1$.
If $d^\pi_\H(x_0,x_1)>\consta$ then for each $i$ there exists $s_i$ such
that $d(\path_{s_i},\qgeo_{t_i})\leqslant \consta$.

Now $\qgeo_{[t_0,t_1]}$ is $\constb$--close to $\path_{[s_0,s_1]}$, which in
turn is $\constb$--close to a subinterval of $g\H$.
Therefore, for each integer $i\in[t_0,t_1]$ there is an integer
$\alpha_i$ such that $d(h^i\!.\bp,gh^{\alpha_i}\inv{g}\!.\bp)\leqslant
2\constb+\constc$.

If for some $i\neq j$ we have
$h^{-i}gh^{\alpha_i}\inv{g}\!.\bp=h^{-j}gh^{\alpha_j}\inv{g}\!.\bp$,
then $h^{j-i}=gh^{\alpha_j-\alpha_i}\inv{g}$, which implies $\langle
h\rangle$ and $\langle gh\inv{g}\rangle$ are commensurable.
However, this
would imply $g\in E(h)$, contrary to hypothesis.
Therefore, for each integer $i$ in $[t_0,t_1]$ we get a distinct point
$h^{-i}gh^{\alpha_i}\inv{g}\!.\bp\in\clnbhd{2\constb+\constc}{\bp}$.
Since the action of $G$ is properly discontinuous, the number of orbit
 points in $\clball{2\constb+\constc}{\bp}$ is finite, so $\diam
 \pi_\H(g\H)$ is bounded, independent of $g$.
\end{proof}

\begin{lemma}[Axiom (P1)]\label{lemma:axiom1}
  For all sufficiently large $\xi$ and for any
  $X,\,Y,\,Z\in\mathbb{Y}$, at most one of $d^\pi_X(Y,Z)$,
  $d^\pi_Y(X,Z)$, and $d^\pi_Z(X,Y)$ is greater than $\xi$.
\end{lemma}
\begin{proof}
  Suppose $\pi_Y$ is $(1,\consta)$--strongly constricting. 
Let $\xi'$ be the constant from
\fullref{lemma:axiom0}.
Let $\xi\geqslant 2\xi'+14\consta$.
Suppose that $d^\pi_X(Y,Z)>\xi$.
We show $d^\pi_X(Y,Z)\leqslant\xi$; the inequality $d^\pi_Z(X,Y)\leqslant\xi$ follows by a
similar argument.

Take any point $z\in Z$, and let $y\in Y$ be a point such that
$d(z,y)=d(z,Y)$.
Let $\geo\from [0,T]\to\X$ be a geodesic from $z$ to $y$.
For every point of $\geo$, $y$ is the closest point of $Y$.
By \fullref{lemma:cpp}, $\pi_Y(\geo)\subset\clball{5\consta}{y}$.
Now, $d^\pi_X(Y,Z)>\xi$ implies $d^\pi_X(\geo_0,\geo_T)>\consta$, so
there is a $z'\in\geo$ and $x\in X$ with $d(x,z')\leqslant\constc$.
By \fullref{lemma:coarse1Lipschitz}, $\pi_Y$ is
$8\consta$--coarsely 1--Lipschitz, which means $d^\pi_Y(x,z')\leqslant 9\consta$.
Thus, $d^\pi_Y(X,Z)\leqslant 2\xi'+d^\pi_Y(x,z)\leqslant
2\xi'+5\consta+d^\pi_Y(x,z')\leqslant 2\xi'+14\consta\leqslant \xi$.
\end{proof}

\begin{lemma}[Axiom (P2)]\label{lemma:axiom2}
  For all sufficiently large $\xi$ and for all $X,\, Y\in\mathbb{Y}$,
  the set $\{V\in\mathbb{Y}\mid d^\pi_V(X,Y)>\xi\}$ is finite. 
\end{lemma}
\begin{proof}
  Let $\xi'$ be the constant of \fullref{lemma:axiom0}.
Suppose $\pi_\H$ is $(1,\consta)$--strongly constricting.
Let $\xi>\consta+2\xi'$. Take arbitrary $X,\,Y\in\mathbb{Y}$, and let
$\geo$ be a geodesic from some point in $\pi_X(Y)$ to
some point in $\pi_Y(X)$.
If $d^\pi_V(X,Y)>\xi$ then $d^\pi_V(\geo_0,\geo_T)>\consta$, so $\geo$
comes within distance $\consta$ of $V$. 
By proper discontinuity of the
action, there are only finitely many elements of $\mathbb{Y}$ that
come within distance $\consta$ of the finite geodesic $\geo$.
\end{proof}

\begin{definition}\label{def:axesquasitree}
  Let $\qt$ be the quasi-tree produced by \fullref{thm:BBF} from
  $\mathbb{Y}$. Let $\qtbp\in\qt$ be the vertex corresponding to
  $\bp\in\X$.
Let $\hat{\pi}_{g\hat{\H}}\from\qt\to g\hat{\H}$ be closest point projection to the
isometrically embedded copy of $g\hat{\H}$ in $\qt$, which the theorem
says coarsely
agrees with $\pi_{g\hat{H}}$.
\end{definition}

\begin{definition}
Define uniform quasi-isometric embeddings $\phi_{g\H}\from g\H\to\qt$
for each translate $g\H$ of $\H$ 
by sending $g\H$ to $f_i\hat{\H}$ via $\phi_{f_i}$, where $g\in
f_iE(h)$, and postcomposing by the isometric embedding of
$f_i\hat{\H}$ into $\qt$ provided by \fullref{thm:BBF}.
\end{definition}

\begin{proposition}
  If there is a strongly contracting element for $G\act\X$ then $G$
  has non-zero growth exponent.
\end{proposition}
\begin{proof}
  \cite[Proposition~3.23]{BesBroFuj10} says $G$ contains a free
  subgroup, so it has exponential growth.
\end{proof}

\section{Abundance of Strongly Contracting Elements}
In this section we show that strongly contracting elements are abundant:
\begin{proposition}\label{lemma:lotsofcontractingelements}
  If $G$ contains a strongly contracting element for $G\act \X$ then so does every
  infinite normal subgroup.
\end{proposition}

In effect, the proposition reduces the problem of growth tightness for arbitrary
quotients of $G$ to quotients by the normal closure of a strongly
contracting element. 

Given a strongly contracting element $h\in\G$ and an infinite normal
subgroup $\Gamma$ of $G$ we find an element
$g\in\Gamma$ such that $f:=gh^ng^{-1}h^{-n}\in\Gamma$ is strongly
contracting for all sufficiently large $n$.
To prove $f$ is strongly contracting we follow a standard strategy by showing that an axis for $f$
has `long' ($\eaddmul n$) segments in contracting sets, separated by
`short' ($=d(\bp,g.\bp)$) hops between such segments.
For each $x\in\X$ there is, coarsely, a unique one of these segments such
that the projection of $x$ transitions from landing at the end of the
segment to landing at the beginning of the segment.
We use this transition point to define the projection to the
$f$--axis, and verify that this projection is strongly contracting. 

We first prove some preliminary lemmas. 

\begin{lemma}\label{lemma:quasigeodesic}
  Let $h\in G$ be an infinite order element and $\pi\from\X\to\langle
  h\rangle.\bp$ a contracting coarse map such that
  $d(x,\pi(x))-d(x,\A)$ is uniformly bounded.
Then $i\mapsto h^i\!.\bp$ is a quasi-geodesic.
\end{lemma}
\begin{proof}
  Take any $\alpha<\beta$ in $\mathbb{Z}$.
By the triangle inequality, $d(h^\alpha\!.\bp,h^\beta\!.\bp)\lmul
(\beta-\alpha)$. We now prove the opposite inequality.
Let $\geo\from [0,T]\to\X$ be a geodesic from $h^\alpha\!.\bp$ to $h^\beta\!.\bp$.
By \fullref{corollary:Morse}, there exists a $\constc$ such that for every $i\in [0,T]\cap\mathbb{Z}$
there exists an $\alpha\leqslant\alpha_i\leqslant\beta$ such that
$d(\geo_i,h^{\alpha_i}\!.\bp)\leqslant\constc$.
Since the action of $G$ on $\X$ is properly discontinuous, there
exists a maximum $\gamma$ such that $d(\bp,h^\gamma\!.\bp)\leqslant
2\constc+1$, so $\alpha_{i+1}-\alpha_i\leqslant \gamma$ for all $i$.
Setting $\alpha_0:=\alpha$ and $\alpha_{\lceil
  T\rceil}:=\beta$, we have $\beta-\alpha=\sum_{i=0}^{\lceil
  T\rceil-1}\alpha_{i+1}-\alpha_{i}\leqslant \gamma \lceil T\rceil\leqslant
\gamma (d(h^\alpha\!.\bp,h^\beta\!.\bp)+1)$.
\end{proof}

 Fix a strongly contracting element $h$, and let $\qt$ be the
  quasi-tree of \fullref{def:axesquasitree}, with bottleneck constant $\bottleneck$.

\begin{lemma}\label{lemma:bigprojectionsamecomponent}
  There exists $K\geqslant 0$ such that $d^\pi_\H(\bp,g_1.\bp)-d^\pi_\H(g_1.\bp,g_0.\bp)\geqslant
  K$ implies $g_0.\qtbp$ and $g_1.\qtbp$ are
  contained in the same component of $\qt\setminus\clball{\bottleneck}{\qtbp}$.
\end{lemma}
\begin{proof}
Let $\constc:=\diam \langle
h\rangle\backslash\hat{\H}$ in $\qt$.
For each $i\in\{0,1\}$, choose an $m_i$ such that
$d(h^{m_i}\!.\qtbp,\hat{\pi}_{\hat{\H}}(g_i.\qtbp))\leqslant \constc$.
Choose a geodesic $\geo$ from $\qtbp$ to $h.\qtbp$.
Take $M>0$ such that
$h^m.\geo\cap\clball{\bottleneck}{\qtbp}=\emptyset$ when $|m|\geq
M$.

For each $i$, $|m_i|\eaddmul d(h^{m_i}\!.\qtbp,\qtbp)\gaddmul K$, so for
sufficiently large $K$ we have
$d(h^{m_i}\!.\qtbp,\qtbp)>2\bottleneck+\constc$ and $|m_i|>M$.
Furthermore, $m_0$ and $m_1$ must have the same sign if $K$ is large
enough: by \fullref{lemma:contractingquasigeodesicsarestable}, the
interval of $\H$ between $h^{m_0}\!.\bp$ and $h^{m_1}\!.\bp$ stays
close to a geodesic between $h^{m_0}\!.\bp$ and $h^{m_1}\!.\bp$, so if
$m_0$ and $m_1$ have different signs:
 \[d^\pi_{{\H}}(g_0.\bp,g_1.\bp)\eadd
 d(h^{m_0}\!.\bp,h^{m_1}\!.\bp)\eadd d(\bp,h^{m_0}\!.\bp)+d(\bp,h^{m_1}\!.\bp)\eadd
d^\pi_{{\H}}(\bp,g_0.\bp)+d^\pi_{{\H}}(\bp,g_1.\bp)\]
However, $d^\pi_{{\H}}(g_0.\bp,g_1.\bp)\leqslant
d^\pi_{{\H}}(\bp,g_1.\bp)-K$, so this would imply \[K\ladd
d^\pi_{{\H}}(\bp,g_0.\bp)\ladd -K,\] which is false for sufficiently large $K$.

No geodesic between $g_i.\qtbp$ and $h^{m_i}\!.\qtbp$ enters
$\clball{\bottleneck}{\qtbp}$, since this would imply:
\[d(h^{m_1}\!.\qtbp,\qtbp)\leqslant 2\bottleneck+\constc\]

For $\min\{m_0,m_1\}\leqslant m\leqslant \max\{m_0,m_1\}-1$ the
geodesic $h^m\!.\geo$ stays outside $\clball{\bottleneck}{\qtbp}$
since $m_0$ and $m_1$ have the same sign and magnitude at least $M$, which implies $|m|\geqslant M$.

By concatenating such geodesics, we construct a path from $g_0.\qtbp$ to $g_1.\qtbp$ in
$\qt\setminus\clball{\bottleneck}{\qtbp}$.
\end{proof}

\begin{corollary}\label{corollary:commoncomponents}
  There exists an $N>0$ such that for all
  $n\geqslant N$ the points $h^n\!.\qtbp$ and $h^N\!.\qtbp$ are in
  the same component of $\qt\setminus\clball{\bottleneck}{\qtbp}$.
\end{corollary}
\begin{proof}
  Take $N$ large enough so that $d^\pi_\H(\bp,h^n\!.\bp)\geqslant
  K+d(\bp,h.\bp)+2C$ for all $n\geqslant N$.
Then
$d^\pi_\H(\bp,h^{n+1}\!.\bp)-d^\pi_\H(h^n\!.\bp,h^{n+1}\!.\bp)\geqslant
K$. Apply \fullref{lemma:bigprojectionsamecomponent}.
\end{proof}

\begin{definition}\label{def:complementarycomponent}
Call the component of $\qt\setminus 
\ball{\bottleneck}{g.\qtbp}$ containing $gh^n\!.\qtbp$ for all
sufficiently large $n$ \emph{the $gh^\infty$
  component} and the component containing $gh^{-n}\!.\qtbp$ for all
sufficiently large $n$ \emph{the $gh^{-\infty}$ component}. 
\end{definition}

\begin{lemma}\label{lemma:disjointballs}
For some $K\geqslant 0$ suppose $g_0$ and $g_1$ are elements of $G$ such that
$g_0\H\neq g_1\H$ and  $d^\pi_{g_0\H}(g_0.\bp,g_1.\bp)\leqslant
K$ and $d^\pi_{g_1\H}(g_0.\bp,g_1.\bp)\leqslant K$.
Then there exists an $N>0$ such that for all $n\geqslant N$, 
  $\epsilon_0,\,\epsilon_1\in\{\pm 1\}$, and $f_0,f_1\in\{g_0,g_1\}$
  \begin{itemize}
  \item the balls 
  $\clball{\bottleneck}{f_0h^{\epsilon_0n/2}\!.\qtbp}$ and 
  $\clball{\bottleneck}{f_1h^{\epsilon_1n/2}\!.\qtbp}$ in $\qt$ are disjoint unless 
  $f_0=f_1$ and $\epsilon_0=\epsilon_1$,
  \item $f_0.\qtbp$ and $f_1.\qtbp$ are in the
    $f_0h^{-\epsilon_0\infty}$ component of
    $\qt\setminus\clball{\bottleneck}{f_0h^{\epsilon_0n/2}\!.\qtbp}$, and
\item $f_0h^{\epsilon_0n}\!.\qtbp$ and $f_0h^{\epsilon_0n}f_1.\qtbp$ are in the
  $f_0h^{\epsilon_0\infty}$ component of
    $\qt\setminus\clball{\bottleneck}{f_0h^{\epsilon_0n/2}\!.\qtbp}$.
  \end{itemize}
 \end{lemma}
 \begin{proof}
$\clball{\bottleneck}{f_0h^{n/2}\!.\qtbp}$ and $\clball{\bottleneck}{f_0h^{-n/2}\!.\qtbp}$ are
    disjoint for all sufficiently large $n$ since $i\mapsto
    h^i\!.\qtbp$ is a quasi-geodesic.
In the other cases, $f_0\H$ and $f_1\H$ are
distinct axes, so $f_0\hat{\H}$ and $f_1\hat{\H}$ are disjoint.
For each $i\in\{0,1\}$, the bounds $d^\pi_{f_i\H}(f_i.\bp,f_{1-i}.\bp)\leqslant
K$ imply that 
the closest point projection $\hat{\pi}_{f_{i}\hat{\H}}(f_{1-i}\hat{\H})$ of $f_{1-i}\hat{\H}$ to
$f_{i}\hat{\H}$ is contained in a bounded neighborhood
of $f_{i}.\qtbp$.
For any point $y.\qtbp\in
\clball{\bottleneck}{f_{1}h^{\epsilon_{1}n/2}\!.\qtbp}\setminus
f_{1}\hat{\H}$, we have that $\hat{\pi}_{f_{1}\hat{\H}}(y\hat{\H})$
is $2\bottleneck$--close to $f_{1}h^{\epsilon_{1}n/2}\!.\qtbp$.
Therefore \[d^\pi_{f_{1}\hat{\H}}(f_0\hat{\H},y\hat{\H})\eadd d(f_{1}.\qtbp,f_{1}h^{\epsilon_{1}n/2}\!.\qtbp)\eaddmul n,\] so for $n$ sufficiently large we can make $d^\pi_{f_{1}\hat{\H}}(f_0\hat{\H},y\hat{\H})$ larger then the constant $\xi$ of
projection axiom (P1), which implies
$d^\pi_{f_0\hat{\H}}(f_{1}\hat{\H},y\hat{\H})\leqslant\xi$.
On the other hand,
$\clball{\bottleneck}{f_{0}h^{\epsilon_{0}n/2}\!.\qtbp}$ projects close to
$f_{0}h^{\epsilon_{0}n/2}\!.\qtbp$ in $f_0\hat{\H}$, so for large enough $n$
the balls have disjoint projections, which means the balls are disjoint.

For the second statement, suppose $N$ is large enough so that for all
$n\geqslant N$ we have $d(\bp,h^{n/2}\!.\bp)\geqslant K'+K+2C$, where $K'$ is the constant of
\fullref{lemma:bigprojectionsamecomponent}.
Then \[d^\pi_{f_0\H}(f_0h^{\epsilon_0 n/2}\!.\bp,f_0.\bp)-d^\pi_{f_0\H}(f_0.\bp,f_{1}.\bp)\geqslant K',\]
so \fullref{lemma:bigprojectionsamecomponent} implies $f_0.\qtbp$ and
$f_1.\qtbp$ are in the same component of
$\qt\setminus\clball{\bottleneck}{f_0h^{\epsilon_0 n/2}\!.\qtbp}$.
If, in addition, $N$ is at least
twice the constant of \fullref{corollary:commoncomponents}, then this
is the $f_0h^{-\epsilon_0\infty}$ component.

The proof of the third statement is similar.
 \end{proof}

\begin{proof}[Proof of \fullref{lemma:lotsofcontractingelements}]
Strongly constricting is the same as strongly contracting, by
\fullref{prop:allthesame}, so suppose $h$ is a
$(1,\consta)$--strongly constricting element.
By \fullref{lemma:cpp}, there exists a $\constc$ such that $\pi_\H$ is
$\constc$--coarsely equivalent to closest point projection.
Recall that $\constc>\consta$.
By \fullref{lemma:coarse1Lipschitz}, there exists a $\constc'$ such that
$\pi_\H$ is $\constc'$--coarsely 1-Lipschitz.

Let $\Gamma$ be an infinite normal subgroup of $G$.
Every infinite order element of $E(h)$ is strongly contracting, so if
$\Gamma$ contains such an element we are done.
Otherwise, $\Gamma\cap E(h)$ is finite.
Since $\Gamma$ is infinite, there exists an element $g\in\Gamma$ such that $g\notin E(h)$. 
We claim that for sufficiently large $n$ the element
$f:=gh^ng^{-1}h^{-n}\in\Gamma$ is strongly constricting.

For brevity, let $f^{i+1/2}$ denote $f^igh^n$.
Let $\hat{\A}_i:=f^{i/2}\hat{\H}$ and $\A_i:=f^{i/2}\H$.
Define $\mathcal{B}_0:=\clball{\bottleneck}{h^{n/2}\!.\qtbp}$,
$\mathcal{B}_1:=\clball{\bottleneck}{f^{1/2}h^{-n/2}\!.\qtbp}$,
  and $\mathcal{B}_{2k+i}:=f^k\mathcal{B}_i$ for $k\in\mathbb{Z}$.
Let $\hat{\zuruck}_i:=f^{i/2}h^{(-1)^in}\!.\qtbp\in\qt$ and
$\zuruck_i:=f^{i/2}h^{(-1)^in}\!.\bp\in\X$.
Let $\hat{\vor}_i:=f^{i/2}\!.\qtbp\in\qt$ and
$\vor_i:=f^{i/2}\!.\bp\in\X$.
See \fullref{fig:qtballs}.

\begin{figure}[h]
\labellist
\small
\pinlabel $\mathcal{B}_{i-1}$ [b] at 0 20
\pinlabel $\mathcal{B}_i$ [t] at 120 40
\pinlabel $\mathcal{B}_{i+1}$ [b] at 200 20
\pinlabel $\mathcal{B}_{i+2}$ [t] at 320 40
\pinlabel $\hat{\A}_i$ [b] at 40 50 
\pinlabel $\hat{\A}_{i+1}$ [t] at 280 10
\pinlabel $\hat{\A}_{i+2}$ [b] at 240 50
\pinlabel $\hat{\A}_{i-1}$ [t] at 80 10
\tiny
\pinlabel $\hat{\zuruck}_i$ [tl] at 62 50
\pinlabel $\hat{\vor}_i$ [tl] at 162 50
\pinlabel $\hat{\zuruck}_{i+2}$ [tl] at 262 50
\pinlabel $\hat{\vor}_{i-1}$ [br] at 59 10
\pinlabel $\hat{\zuruck}_{i+1}$ [br] at 159 10
\pinlabel $\hat{\vor}_{i+1}$ [br] at 259 10 
\endlabellist
  \centering
  \includegraphics{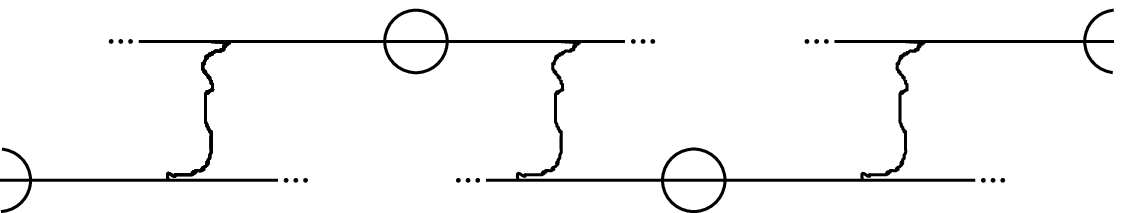}
  \caption{Disjoint balls in $\qt$.}
  \label{fig:qtballs}
\end{figure}

By repeated applications of \fullref{lemma:disjointballs},
for large enough $n$ the balls $\mathcal{B}_i$ are pairwise disjoint.
There are two orbits of these balls under the $f$--action, so $f$ is
an infinite order element.
Furthermore, the balls are linearly ordered by separation, consistent with
the subscripts, since for all $i$ we have that $\mathcal{B}_j$ is contained in
the $f^{i/2}h^{(-1)^{i+1}\infty}$ component of
$\qt\setminus\mathcal{B}_i$ for all $j>i$, and in the
$f^{i/2}h^{(-1)^{i}\infty}$ component for all $j<i$.

For any $i$ and any $j<i-1$ the ball $\mathcal{B}_{i-1}$ separates
$\hat{\A}_j$ from $\hat{\A}_i$ in $\qt$, so
$\hat{\pi}_{\hat{\A}_i}(\hat{\A}_j)$ is contained in a bounded
neighborhood of $\hat{\pi}_{\hat{\A}_i}(\hat{\A}_{i-1})$, which in
turn we know is contained in a bounded neighborhood of
$\hat{\zuruck}_i$.
Conversely, $\hat{\pi}_{\hat{\A}_i}(\hat{\A}_j)$ is contained in
a bounded neighborhood of $\hat{\vor}_i$ for $j>i$.
Since $\hat{\pi}_{\hat{\A}_i}$ agrees with $\pi_{\A_i}$ up to
bounded error, the same statements are true for the axes in $\X$. 
That is, there exists a $\constb$, independent of $n$, such that for all $i$
we have
\begin{itemize} 
\item $d^\pi_{\A_i}(\zuruck_i, \A_j)\leqslant \constb$ if $j<i$, and
\item $d^\pi_{\A_i}(\vor_i, \A_j)\leqslant \constb$ if $j>i$.
\end{itemize}
Define $\constb':=2\constb+\consta+2\constc+\constc'$.

Suppose that for some $x\in\X$ there exists an $i$ such that
$d^\pi_{\A_i}(x,\vor_i)>\constb'$.
Then for any $j>i$ we have $d(\pi_{\A_i}(x),\pi_{\A_i}(\A_j))>\constc>\consta$.
Let $y$ be a point of $\A_j$ closest to $x$.
On any given geodesic from $x$ to $y$ there is a point
$z\in\clball{\consta+\constb}{\vor_i}$, since $d_{\A_i}^\pi(x,y)>\consta$.
Now $\pi_{\A_j}$ is $\constc$--coarsely equivalent to closest point
projection, and $y$ is closest to both $x$ and $z$, so
$d^\pi_{\A_j}(x,z)\leqslant 2\constc$.
However, $z$ is $(\consta+\constb)$--close to $\vor_i$, and
$d^\pi_{\A_j}(\vor_i,\zuruck_j)\leqslant \constb$, so
$d^\pi_{\A_j}(x,\zuruck_j)\leqslant 2\constc+\consta+\constb+\constc'+\constb=\constb'$.

We have shown that $d^\pi_{\A_i}(x,\vor_i)>\constb'$ implies
$d^\pi_{\A_j}(x,\zuruck_j)\leqslant \constb'$ for all $j>i$.
A similar argument shows that $d^\pi_{\A_i}(x,\zuruck_i)>\constb'$ implies
$d^\pi_{\A_j}(x,\vor_j)\leqslant \constb'$ for all $j<i$.

Assume that $n$ is large enough so that 
$d^\pi_{\A_0}(\zuruck_0,\vor_0)=d^\pi_\H(h^n\!.\bp,\bp)>2\constb'+2\consta+2\constc+d(\bp,g.\bp)$.
Define $\mathcal{F}:=\cup_{i\in\mathbb{Z}}\{\vor_i\}$.
We wish to define $\pi_\mathcal{F}\from
\X\to\mathcal{F}$ by sending a point $x$ to the point $\vor_\alpha$
where $\alpha$ is the greatest integer such that
$d^\pi_{\A_\alpha}(x,\vor_{\alpha})\leqslant \constb'$, but we must
verify that such an $\alpha$ exists.
Fix an $x\in\X$, and suppose that $\iota\in\mathbb{Z}$ is such that
$d(x,\A_\iota)=\min_{j\in\mathbb{Z}} d(x,\A_j)$. 
Such an $\iota$ exists since the action is properly discontinuous.
Suppose that $d^\pi_{\A_{\iota}}(x,\vor_{\iota})\leqslant \constb'$. 
By the assumption on $n$, $d^\pi_{\A_{\iota}}(x,\zuruck_{\iota})> \constb'$, so
$d^\pi_{\A_j}(x,\vor_{\iota})\leqslant \constb'$ for all $j<\iota$.
A brief computation shows that
$d^\pi_{\A_{\iota+1}}(x,\zuruck_{\iota+1})\leqslant d(x,\A_{\iota+1})+d(\bp,g.\bp)+\constb'+2\consta+\constc$.
By \fullref{lemma:cpp}, $d(\zuruck_{\iota+1},\pi_{\A_{\iota+1}}(x))\leqslant
d(\bp,g.\bp)+\constb'+2\consta+2\constc$, which, again by our assumption on
$n$, implies $d^\pi_{\A_{\iota+1}}(x,\vor_{\iota+1})> \constb'$.
We conclude that $\alpha\leqslant\iota$.
The previous paragraph then tells us that $d^\pi_{\A_j}(x,\zuruck_j)\leqslant \constb'$ for all $j>\alpha+1$.

\begin{figure}[h]
\labellist
\tiny
\pinlabel $\ball{\constb'}{\zuruck_{i+1}}$ [tl] at 72 50
\pinlabel $\ball{\constb'}{\vor_{i+1}}$ [tl] at 172 50
\pinlabel $\ball{\constb'}{\zuruck_{i+3}}$ [tl] at 272 50
\pinlabel $\ball{\constb'}{\vor_{i}}$ [br] at 39 10
\pinlabel $\ball{\constb'}{\zuruck_{i+2}}$ [br] at 139 10
\pinlabel $\ball{\constb'}{\vor_{i+2}}$ [br] at 239 10 
\pinlabel \textcircled{x} [t] at 45 9 
\pinlabel \textcircled{y} [t] at 52 9
\pinlabel \textcircled{x} [t] at 150 9 
\pinlabel \textcircled{y} [t] at 247 9
\pinlabel \textcircled{x} [b] at 150 47 
\pinlabel \textcircled{y} [b] at 165 47
\pinlabel \textcircled{x} [b] at 261 47 
\pinlabel \textcircled{y} [b] at 268 47
\endlabellist
  \centering
  \includegraphics{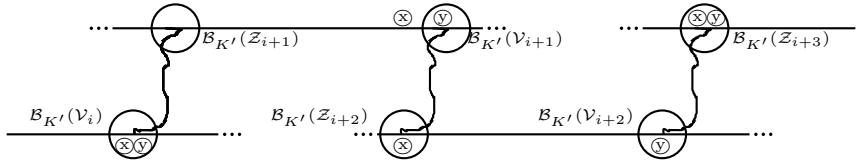}
  \caption{Projections \textcircled{x} of $x$ and \textcircled{y} of $y$ to each axis.}
  \label{fig:iplus2}
\end{figure}

Now suppose $x$ and $y$ are points with $\pi_\mathcal{F}(x)=\vor_i$ and $\pi_\mathcal{F}(y)=\vor_j$ for
$j>i+1$. 
Then for each $i+2\leqslant k\leqslant j$ we have $d^\pi_{\A_k}(x,y)\geqslant
d^\pi_{\A_k}(\zuruck_k,\vor_k)-2\constb'>\consta$. 
\fullref{fig:iplus2} depicts a situation with $j=i+2$ that shows $j>i+1$
is necessary, since the projections to $\A_{i+1}$ may be close.
By the strong constriction property for each $\A_k$, every
geodesic from $x$ to $y$ passes $(\consta+\constb')$--close to $\zuruck_k$ and $\vor_k$.
So every geodesic passes within $\consta+\constb'$ of  $\pi_\mathcal{F}(y)=\vor_j$ and
within $\consta+\constb'$ of $\zuruck_{i+2}$, which boundedly close to
$\pi_\mathcal{F}(x)=\vor_i$.

Therefore, $\pi_\mathcal{F}$ is $(1, \max\{d(\vor_0,\vor_2),
\consta+\constb'+d(\vor_0,\zuruck_2)\})$--strongly constricting.
\fullref{lemma:quasigeodesic} says $i\mapsto f^i\!.\bp$ is a
quasi-geodesic, so $f\in\Gamma$ is a strongly contracting element.
\end{proof}

\section{A Minimal Section}
Let $\X$ be a $G$--space with basepoint $\bp$. 
Suppose that there exists a strongly contracting element for $G\act\X$.
Let $\Gamma$ be an
infinite normal subgroup of $G$.
By \fullref{lemma:lotsofcontractingelements}, there exists a strongly
contracting element $h\in \Gamma$.
Let $\H=E(h).\bp$ be an axis for $h$, and define equivariant
projections to translates of $\H$ as in 
\fullref{def:equivariantprojections}.
Suppose $\pi_\H$ is a $(1,\consta)$--strongly constricting $\consta$--coarse map.

\begin{definition}
For each element $g\Gamma\in G/\Gamma$ choose an element $\bar{g}\in
g\Gamma$ such that $d(\bp,\bar{g}.\bp)= d(\bp,g\Gamma.\bp)=d(\Gamma.\bp,g\Gamma.\bp)$.
Let $\bar{G}:=\{\bar{g}\mid g\Gamma\in G/\Gamma\}$.
We call $\bar{G}$ a \emph{minimal section}, and let $\bar{\G}$ denote $\bar{G}.\bp$.
\end{definition}

Observe that
$\Theta'_{G/\Gamma}(s)=\Theta'_{\bar{G}}(s)$, so $\rate_{G/\Gamma}=\rate_{\bar{G}}$.

The next lemma says, coarsely, that the minimal section is orthogonal
to translates of $\H$.

\begin{lemma}\label{lemma:boundedprojection}
 For every $\bar{g}\in\bar{G}$
 and for every $f\in G$ we have $d_{f\H}^\pi(\bp,\bar{g}.\bp)\leqslant
 8\consta+\constc$, where $\constc:=\diam \langle h\rangle\backslash \H$.
\end{lemma}
\begin{proof}
Suppose not.
Then there exists
an $n\neq 0$ such that:
\begin{align*}
   D&\geqslant
   d(\pi_{f\H}(\bp),fh^nf^{-1}.\pi_{f\H}(\bar{g}.\bp))\\
&\geqslant d^\pi_{f\H}(\bp,fh^nf^{-1}\bar{g}.\bp)-2\consta
\end{align*}
Thus, $d^\pi_{f\H}(\bp,\bar{g}.\bp)-d^\pi_{f\H}(\bp,fh^n\inv{f}\bar{g}.\bp)>6\consta$.
However:
\begin{align*}
  d(\bp,fh^n&f^{-1}\bar{g}.\bp)\\
&\leqslant
d(\bp,\pi_{f\H}(\bp))+d^\pi_{f\H}(\bp,fh^n\inv{f}\bar{g}.\bp)+d(\pi_{f\H}(fh^n\inv{f}\bar{g}.\bp),fh^n\inv{f}\bar{g}.\bp)\\
&<
d(\bp,\pi_{f\H}(\bp))+d^\pi_{f\H}(\bp,\bar{g}.\bp)+d(\pi_{f\H}(fh^n\inv{f}\bar{g}.\bp),fh^n\inv{f}\bar{g}.\bp)-6\consta\\
&= d(\bp,\pi_{f\H}(\bp))+d^\pi_{f\H}(\bp,\bar{g}.\bp)+d(\pi_{f\H}(\bar{g}.\bp),\bar{g}.\bp)-6\consta\\
&\leqslant d(\bp,\bar{g}.\bp) \quad\text{(by \fullref{lemma:distancelb})}
\end{align*}
This contradicts minimality of $\bar{G}$, since $fh^nf^{-1}\bar{g}=\bar{g}\inv{\bar{g}}fh^n\inv{f}\bar{g}\in\bar{g}\Gamma$.
\end{proof}

\begin{corollary}\label{corollary:maxsepdist}
If
  $d(\bar{g}.\bp,\bar{g}'\!.\bp)\geqslant 18\consta+2\constc$ for
  $\bar{g},\,\bar{g}'\in\bar{G}$ then there is no $f\in G$
  such that $\bar{g}.\bp\in f\H$ and $\bar{g}'\!.\bp\in f\H$.
\end{corollary}
\begin{proof}
  If there were such an $f$, we would have
  $d^\pi_{f\H}(\bar{g}.\bp,\bar{g}'\!.\bp)\geqslant 2(8\consta+\constc)$,
  which means either $\bar{g}$ or $\bar{g}'$ would contradict \fullref{lemma:boundedprojection}.
\end{proof}

In light of \fullref{corollary:maxsepdist}, it will be convenient to
pass to a coarsely dense subset of $\bar{\G}$ whose elements yield
distinct translates of $\H$:

\begin{definition}\label{def:A}
Let $\constb\geqslant 18\consta+2\constc$, and let $A$ be a maximal subset of $\bar{G}$
such that $1\in A$ and $d(\bar{g}.\bp,\bar{g}'\!.\bp)\geqslant \constb$ for
all distinct $\bar{g},\,\bar{g}'\in A$.
Let $\A:=A.\bp$.
\end{definition}
By maximality, for every $\bar{g}\in\bar{G}$ there is some $a\in A$ such that
$d(a.\bp,\bar{g}.\bp)\leqslant \constb$.
There are boundedly many points of $\bar{\G}$ in a ball of
radius $\constb$, so $\Theta_{\bar{\G}}(s)$ is bounded below by
$\Theta_{\A}(s)$ and above by a
constant multiple of $\Theta_{\A}(s)$.
In particular, $\Theta_{\A}(s)$ has the same convergence behavior as
$\Theta_{\bar{\G}}(s)$, so $\rate_A=\rate_{\bar{G}}=\rate_{G/\Gamma}$.

\fullref{corollary:maxsepdist} implies $a\H\neq a'\H$ for distinct $a,\, a'\in A$.

\section{Embedding a Free Product Set}

Let $A\subset\bar{G}$ as in \fullref{def:A}, and let
$A^*:=A\setminus\{1\}$.
Consider the free product set 
 $A^*\!*\mathbb{Z}_2:=\bigcup_{k=1}^\infty\{(a_1,\dots,a_k)\mid
a_i\in A^*\}$.
For any $n>0$ we can map the free product set into $G$ by $\phi_n\from(a_1,\dots, a_k)\mapsto
a_1h^na_2h^n\!\cdots a_kh^n$.
Our goal is to show $\rate_{\phi_n(A^*\!*\mathbb{Z}_2)}>\rate_A$.
We establish the inequality in the next section.
In this section we show $\phi_n$ is an injection for all sufficiently
large $n$. 
In fact, we prove something stronger:

\begin{proposition}\label{proposition:embedded}
The map $A^*\!*\mathbb{Z}_2\to G: (a_1,\dots, a_k)\mapsto a_1h^n\!\cdots a_nh^n\!.\bp$
is an injection for all sufficiently large $n$.
\end{proposition}
The map is an injection because we have an action of $G$ on the
quasi-tree $\qt$, and for large enough $n$ we have ``quasi-edges'' of
the form
$[y,yh^n]$.
We have set things up so that the $a$'s do not backtrack across such
edges. See \fullref{fig:embed}.
We make this precise:
\begin{proof}
Let $\underline{a}=(a_1,\dots,a_k)\in A^**\mathbb{Z}_2$.

By \fullref{lemma:boundedprojection}, there is a $\constb$ such that 
$d_{f\H}^\pi(\bp,\bar{g}.\bp)\leqslant K$ for every $f\in G$ and every
$\bar{g}\in\bar{G}$.
The choice of $A\subset\bar{G}$ in
\fullref{def:A} guarantees that the axes $a\H$ for $a\in A$
are distinct.
Let $N$ be the constant of \fullref{lemma:disjointballs} for this $K$,
and choose $n\geqslant N$.

Note that the proof of \fullref{lemma:disjointballs} includes the fact that $d(\bp,h^{n/2}\!.\bp)\geqslant K'+K+2C$, where $K'$ is the constant of
\fullref{lemma:bigprojectionsamecomponent}.
Therefore, if $\phi_n(\underline{a}).\bp=\phi_n(\underline{a}').\bp$ then
\[d^\pi_{\phi_n(\underline{a})\H}(\phi_n(\underline{a}).\bp,\phi_n(\underline{a})h^{-n/2}\!.\bp)-d^\pi_{\phi_n(\underline{a})\H}(\phi_n(\underline{a}).\bp,\phi_n(\underline{a}').\bp)\geq
K'+C>K',\] so \fullref{lemma:bigprojectionsamecomponent} implies
$\phi_n(\underline{a}).\qtbp$ and $\phi_n(\underline{a}').\qtbp$,
though they might not be equal, are at least
contained in the same component of $\qt\setminus\clball{\bottleneck}{\phi_n(\underline{a})h^{-n/2}\!.\qtbp}$.

\begin{figure}[h]
  \centering
\labellist
\small
\pinlabel $h^n\!.\qtbp$ [l] at 126 16
\pinlabel $\qtbp$ [t] at 62 14
\pinlabel $a_1.\qtbp$ [b] at 98 102
\pinlabel $a_1'.\qtbp$ [br] at 74 158
\pinlabel $a_1h^na_2.\qtbp$ [br] at 230 173
\pinlabel $a_1h^n\!.\qtbp$ [tl] at 196 104
\pinlabel $a_1h^na_2h^n\!.\qtbp$ [l] at 286 175
\pinlabel $a_1'h^n\!.\qtbp$ [t] at 148 160
\tiny
\pinlabel $\clball{\bottleneck}{a_1h^na_2h^{n/2}\!.\qtbp}$ [b] at 253 192
\pinlabel $\clball{\bottleneck}{h^{n/2}\!.\qtbp}$ [bl] at 104 28
\pinlabel $\clball{\bottleneck}{h^{-n/2}\!.\qtbp}$ [br] at 19 28
\pinlabel $\clball{\bottleneck}{a_1'h^{n/2}\!.\qtbp}$ [b] at 99 177
\pinlabel $\clball{\bottleneck}{a_1h^{n/2}\!.\qtbp}$ [t] at 154 91
\endlabellist
  \includegraphics[height=2in]{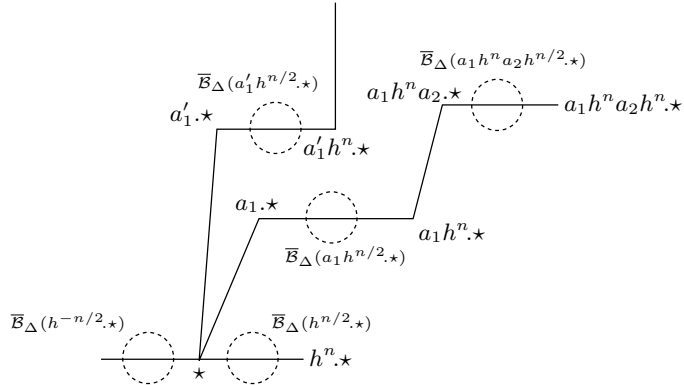}
  \caption{$A$ does not cross $h^n$ quasi-edges}
  \label{fig:embed}
\end{figure}

Define $\vor_i(\underline{a})$ to be the $a_1h^n\cdots a_ih^{\infty}$ component of
$\qt\setminus\clball{\bottleneck}{a_1h^n\cdots a_ih^{n/2}\!.\qtbp}$
for $i\leqslant k$ (recall \fullref{def:complementarycomponent}). 
\fullref{lemma:disjointballs} implies that
$\vor_i(\underline{a})\supset\vor_{i+1}(\underline{a})$ and
$\phi_n(\underline{a}).\qtbp\in\vor_k(\underline{a})$.
Moreover, for $i\leq\min\{k,k'\}$,
$\vor_i(\underline{a})$ and $\vor_i(\underline{a}')$ are disjoint
  unless $a_j=a_j'$ for all $j\leqslant i$.

If $\phi_n(\underline{a}).\bp=\bp$ then 
\fullref{lemma:bigprojectionsamecomponent}
implies $\qtbp\in\vor_k(\underline{a})\subset\vor_1(\underline{a})$.
This contradicts the fact that $\qtbp$ is contained in the
$a_{1}h^{-\infty}$ component of
$\qt\setminus\clball{\bottleneck}{a_{1}h^{n/2}\!.\qtbp}$.
The same argument shows that if
 $\underline{a}$ is a proper prefix of
$\underline{a}'$, that is, if $\underline{a}=(a_1,\dots,a_k)$ and
$\underline{a}'=(a_1,\dots,a_k,a'_{k+1},\dots,a'_{k'})$ with $k'>k$,
then $\phi_n(\underline{a}).\bp\neq\phi_n(\underline{a}').\bp$.

Suppose $\phi_n(\underline{a}).\bp=\phi_n(\underline{a}').\bp$ with
$k\leqslant k'$.
\fullref{lemma:bigprojectionsamecomponent} implies
$\phi_n(\underline{a}).\qtbp
\in\vor_{k'}(\underline{a}')$,
so $a_i=a_i'$ for all $i\leqslant k$.
Since $\underline{a}$ cannot be a proper prefix of $\underline{a}'$, $k=k'$.
Hence, $\phi_n(\underline{a}).\bp=\phi_n(\underline{a}').\bp$ implies
$\underline{a}=\underline{a}'$ for all sufficiently large $n$.
\end{proof}

\section{Growth Gap}
A free product of groups has greater growth exponent than the factor
groups, with respect to a word metric, so we expect that
$\phi_n(A^*\!*\mathbb{Z}_2)$  should have a larger growth exponent
than $A$.
To verify this intuition, one must show that the Poincar\'e series for
$\phi_n(A^*\!*\mathbb{Z}_2)$ diverges at $\rate_A+\epsilon$ for some $\epsilon>0$.
A clever manipulation of Poincar\'e series yields the following criterion:
\begin{lemma}[{\cite[Criterion 2.4]{DalPeiPic11},\cite[Proposition 2.3]{Sam02amalgam}}]\label{lemma:strictgrowthinequality}
  If the map \[\phi_n\from A^*\!*\mathbb{Z}_2\to G:(a_1,\dots, a_k)\mapsto a_1h^n\!\cdots
  a_kh^n\] is an injection, and if
$\exp(|h^n|\cdot
  \rate_A)<\Theta_{A}(\rate_A)$,
 then $\rate_{\phi_n(A^*\!*\mathbb{Z}_2)}>\rate_{A}$.
\end{lemma}

Because our methods are coarse we have passed to a high power $h^n$ of $h$ and
therefore do not have control over $|h^n|$.
However, the criterion is satisfied automatically if $A$, or, equivalently,
$\bar{G}$, is divergent, which, recalling \fullref{def:divergent}, means
$\Theta_{A}$ diverges at $\rate_{A}$.
The following definition will be used in a condition to guarantee
divergence of $\bar{G}$.

\begin{definition}\label{def:cge}
Let $Comp_{Q,\,r}^{G}\subset G.\bp$ be the set of points $g.\bp$ such
that there exists a geodesic $[x,y]$ of length $r$ with $x\in\clball{Q}{\bp}$ and
$y\in\clball{Q}{g.\bp}$ whose interior is contained in $\thin$.

 Define the $Q$--\emph{complementary growth exponent} of $G$ to be:
\[\rate_G^c:=\limsup_{r\to\infty}\frac{\log
  \# Comp_{Q,\,r}^{G}}{r}\]
\end{definition}

\begin{theorem}\label{maincge}
Let $G$ be a finitely generated, non-elementary group.
  Let $\X$ be a $G$--space.
If $G$ contains a strongly contracting element and there exists a
$Q\geqslant 0$ such that the $Q$--complementary growth exponent of $G$ is strictly less than the
  growth exponent of $G$, then $G\act \X$ is a growth tight action.
\end{theorem}
  The proof of \fullref{maincge} follows in part the proof of
  \cite[Theorem~1.4]{DalPeiPic11} for geometrically finite Kleinian
groups.
For the divergence part of the proof, the Kleinian group
ingredients of \cite[Theorem~1.4]{DalPeiPic11} are inessential, and
our changes are mostly cosmetic.
The real generalization is in the use of
\fullref{proposition:embedded} instead of a Ping-Pong argument.
\begin{proof}
  Let $\Gamma$ be an infinite, infinite index normal subgroup of $G$.
By \fullref{lemma:lotsofcontractingelements}, there is a strongly contracting
element in $\Gamma$.
Let $\bar{G}$ be a minimal section of $G/\Gamma$.
If $\rate_{\bar{G}}\leqslant\rate_G^c$ then we are done, since
$\rate_G^c<\rate_G$, so suppose $\rate_{\bar{G}}>\rate_G^c$.
\begin{claim}
  $\bar{G}$ is divergent.
\end{claim}
Assume the claim, and let $A$ be a maximal separated set in $\bar{G}$ as in \fullref{def:A}.
Then $A$ and $\bar{G}$ have the same critical exponent, and are both divergent.
By \fullref{proposition:embedded}, for sufficiently large $n$ the map
$\phi_n\from A^*\!*\mathbb{Z}_2\to G$ is an injection.
By \fullref{lemma:strictgrowthinequality},
$\rate_A<\rate_{\phi_n(A^*\!*\mathbb{Z})}$. Thus,
$\rate_{G/\Gamma}=\rate_\A<\rate_{\phi_n(A^*\!*\mathbb{Z}_2)}\leqslant\rate_G$.

It remains to prove the claim.

Let $r>0$, and suppose $d(\bp,\bar{g}.\bp)=r$.
Let
$0\leqslant M_0\leqslant r$ and $M_1=r-M_0$.
Choose a geodesic $[\bp,\bar{g}.\bp]$ from $\bp$ to $\bar{g}.\bp$, and
let $[\bp,\bar{g}.\bp](M_0)$ denote the point of $[\bp,\bar{g}.\bp]$ at distance
$M_0$ from $\bp$.

\begin{figure}[h]
  \centering
\labellist
\pinlabel $\ball{Q}{\bp}$ at 23 60
\pinlabel $\bp$ [tr] at 40 40
\pinlabel $\bar{g}.\bp$ [tl] at 288 40
\pinlabel $g_0.\bp$ [tl] at 97 40
\pinlabel $g_1.\bp$ [tr] at 232 40
\pinlabel $x_0$ [b] at 109 76
\pinlabel $x_1$ [b] at 219 76
\pinlabel $[\bp,\bar{g}.\bp]$ [b] at 151 86
\endlabellist
  \includegraphics{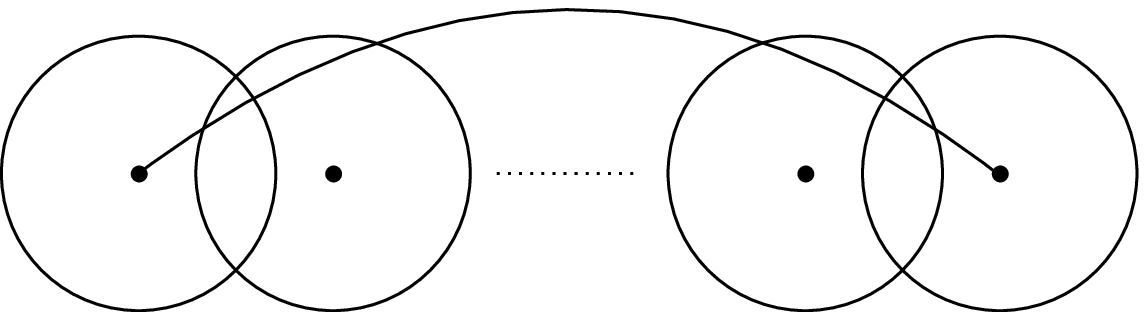}
  \caption{Splitting a geodesic into three subsegments}
  \label{fig:3subsegments}
\end{figure}

First, we suppose that $[\bp,\bar{g}.\bp](M_0)\in\thin$.
Let $[x_0,x_1]\subset [\bp,\bar{g}.\bp]$ be the largest subsegment
containing  $[\bp,\bar{g}.\bp](M_0)$
such that $(x_0,x_1)\subset\thin$.
Let $m_0=d(\bp,x_0)$, and let $m_1=d(x_1,\bar{g}.\bp)$.
There exist group elements $g_i\in G$ such that $d(g_i.\bp,x_i)\leqslant
Q$. See \fullref{fig:3subsegments}.
We have $\bar{g}.\bp=g_0\cdot \inv{g}_0g_1\cdot \inv{g}_1\bar{g}.\bp$.
Now $m_0-Q\leqslant d(\bp,\bar{g}_0.\bp)\leqslant d(\bp,g_0.\bp)\leqslant m_0+Q$, and
$m_1-Q\leqslant d(\bp,\overline{g^{-1}_1\bar{g}}.\bp)\leqslant
d(\bp,g^{-1}_1\bar{g}.\bp)\leqslant m_1+Q$.
Furthermore, $\inv{g}_0g_1\in Comp_{Q,\,r-(m_0+m_1)}^{G}$.
Thus, the point $\bar{g}.\bp$ can be expressed as  a product of an element of $\bar{G}$
of length $m_0\pm Q$, an element of $\bar{G}$ of length $m_1\pm Q$,
and the quotient of an element of $Comp_{Q,\,r-(m_0+m_1)}^{G}$.

\medskip
(\textdagger)\hfill\parbox{.9\textwidth}{
  The same is also true if $[\bp,\bar{g}.\bp](M_0)\in\thick$, in
which case we can take $m_0=M_0$ and $m_1=r-m_0$.
Then choose $g_0=g_1$ so that the
contribution from 
$Comp_{Q,\,r-(m_0+m_1)}^{G}$ is trivial.}
\smallskip

Let $V_{r,Q}:=\#\left(\bar{G}.\bp\cap \clnbhd{r+Q}{\bp}\setminus \ball{r-Q}{\bp}\right)$.
For every $M_0+M_1=r$ we have:
\[V_{r,Q}\lmul\sum_{m_0=0}^{M_0}
\sum_{m_1=0}^{M_1}V_{m_0, Q} \cdot V_{m_1, Q} \cdot \#
Comp_{Q,\,r-(m_0+m_1)}^{G}\]

Choose $\xi>0$ such that $\rate_{\bar{G}}\geqslant 2\xi+\rate_G^c$.
Since \[\# Comp_{Q,\,r-(m_0+m_1)}^{G}\lmul
\exp((r-(m_0+m_1))(\rate_{\bar{G}} -\xi))\] whenever $r-(m_0+m_1)$ is
sufficiently large, it follows that:
\begin{multline}
  \label{eq:2}
  V_{r,Q}\cdot \exp(-r(\rate_{\bar{G}}-\xi))\lmul\\
\left(\sum_{m_0=0}^{M_0} V_{m_0,
  Q}\cdot \exp(-m_0(\rate_{\bar{G}}-\xi))
\right)\cdot\left(\sum_{m_1=0}^{M_1} V_{m_1, Q} \cdot
  \exp(-m_1(\rate_{\bar{G}}-\xi))  \right)
\end{multline}
Set $w_i:=V_{i,Q}\cdot \exp(-i(\rate_{\bar{G}}-\xi))$ and
$W_i:=\sum_{j=1}^iw_i$. 
Then \eqref{eq:2} and \cite[Lemma~4.3]{DalPeiPic11} imply that
$\sum_i w_i\cdot\exp(-is)$ diverges at its critical exponent, which is:
\[\limsup_i\frac{\log w_i}{i}=\left(\limsup_i\frac{\log V_{i,Q}}{i}\right)- (\rate_{\bar{G}}-\xi)=\xi\]
So $\infty=\sum_iw_i\cdot\exp(i\xi)=\sum_iV_{i,Q}\cdot\exp(-i\rate_{\bar{G}})\emul \Theta_{\bar{G}}(\rate_{\bar{G}})$.
\end{proof}

\begin{theorem}\label{mainqc}
Let $G$ be a finitely generated, non-elementary group.
Let $\X$ be a quasi-convex $G$--space.
If $G$ contains a strongly contracting element
then $G\act \X$ is a growth tight action.
\end{theorem}
\begin{proof}
The proof is an easier special case of the proof of \fullref{maincge}.
If $\X$ is $Q$--quasi-convex then we can always choose to be in case (\textdagger) of
the proof.
\end{proof}

\section{Growth of Conjugacy Classes}\label{sec:cc}
Parkkonen and Paulin \cite{ParPau13} ask: given a finitely generated group $G$ with a word metric and an element $h\in G$, what is growth rate of the conjugacy class $[h]$ of $h$?
In a hyperbolic group $G$ there is a finite subgroup, the
\emph{virtual center}, consisting of elements whose centralizer is finite index in $G$. 
The growth exponent of a conjugacy class in the virtual center is clearly zero.
Parkkonen and Paulin show that for every element $h$ not in the virtual center, $\rate_{[h]}=\frac{1}{2}\rate_G$.
This generalized an old result of Huber \cite{Hub56} for the case of $G$ acting cocompactly on the hyperbolic plane and $h$ loxodromic.

Since strongly contracting elements behave much like infinite order
elements in hyperbolic groups, it is natural to ask whether the growth
exponent of the conjugacy class of a strongly contracting element $h$ also
satisfies $\rate_{[h]}=\frac{1}{2}\rate_G$.

We show that the lower bound holds, and the upper bound holds if $h$
moves the basepoint sufficiently far with respect to the contraction
constant for the axis. 

\begin{theorem}\label{thm:conjugacyclassgrowth}
Let $G$ be a non-elementary, finitely generated group, and let $\X$ be a $G$--space.
Let $h$ be a strongly contracting element for $G\act\X$.
Then $\rate_{[h]}\geqslant\frac{1}{2}\rate_{G}$.

Let $\constc:=\diam \gpcenter(h)
\backslash \H$, where $\gpcenter(h)$ is the centralizer of $h$ in $G$.
Suppose $\pi_\H$ is a $(1,\consta)$--strongly constricting,
$\consta$--coarse map. 
If $d(\bp,h.\bp)>15\consta+2\constc$ then $\rate_{[h]}=\frac{1}{2}\rate_{G}$.
\end{theorem}

\begin{corollary}
  For $h$ strongly contracting, $\rate_{[h^n]}=\frac{1}{2}\rate_{G}$ for all sufficiently large $n$.
\end{corollary}
\begin{proof}
For $n$ nonzero, $E(h^n)=E(h)$ and
  $\gpcenter(h^n)\supset\gpcenter(h)$, so the same $\consta$ and
  $\constc$ work for $h^n$ as work for $h$.
On the other hand,  $\langle h\rangle$ is quasi-isometrically embedded,
  so $d(\bp,h^n\!.\bp)\eaddmul n$.
Thus, $d(\bp,h^n\!.\bp)>15\consta+2\constc$ for large enough $n$.
\end{proof}
It would be interesting to know whether the restriction on $d(\bp,h.\bp)$ is
really necessary:

\begin{question}
  Does there exist an action $G\act \X$ such that $h$ is a strongly
  contracting element with $\rate_{[h]}>\frac{1}{2}\rate_G$?
\end{question}

\begin{proof}[Proof of \fullref{thm:conjugacyclassgrowth}]
Define $\constb:=6\consta+\constc$ and $F:=\{g\in G\mid
d^\pi_{g\H}(\bp,g.\bp)\leqslant \constb\}$.

First, we will show $\rate_F=\rate_G$. Then, we will relate
$\rate_{[h]}$ to $\rate_F$.

For any $r\geqslant 0$ consider $\phi\from \{f\in F\setminus E(h)\mid d(\bp,f.\bp)\leqslant r\}\to \{g\H \mid g\in G\setminus E(h)
\text{ and }g\H\cap \clball{r}{\bp}\neq\emptyset\}$ defined by
$\phi(f):=f\H$.
For each axis $g\H$ meeting $\clball{r}{\bp}$ there exists a $g'\in
gE(h)$ such that $d(\bp,g'\!.\bp)=d(\bp,g\H)\leqslant r$.
Since $\pi_{g\H}$ is within $5\consta$ of closest point projection, by
\fullref{lemma:cpp}, we have $d^\pi_{g'\H}(\bp,g'\!.\bp)\leqslant 6\consta\leqslant
\constb$. 
Therefore, $g'\in F$ with $\phi(g')=g\H$, so $\phi$ is surjective. 

We estimate:
\[\#\text{axes meeting $\clball{r}{\bp}$} \geqslant
\frac{|G.\bp\cap\clball{r}{\bp}|\times \#\text{axes per orbit
    point}}{\text{maximum number of orbit points per axis}}\]
The basepoint belongs to
$[\mathrm{Stab}_G(\bp):E(h)\cap\mathrm{Stab}_G(\bp)]$ distinct
translates of $\H$, so the number of axes per orbit point is 
constant.
The maximum number of orbit points in $\clball{r}{\bp}$ contained in a
single axis is proportional to $r$, since each axis is a
quasi-isometrically embedded image of a virtually cyclic group.
Combined with surjectivity of $\phi$ this gives:
 \[|F.\bp\cap\clball{r}{\bp}| \gmul \frac{|G.\bp \cap\clball{r}{\bp}|}{r}\]
Thus:
\begin{align*}
\rate_F&=\limsup_{r\to\infty}\frac{1}{r}\log 
  |F.\bp\cap\clball{r}{\bp}|\\
&\geqslant \limsup_{r\to\infty}\frac{1}{r}\log \frac{|G.\bp 
  \cap\clball{r}{\bp}|}{r}\\
&=\limsup_{r\to\infty}\frac{1}{r}\log |G.\bp \cap\clball{r}{\bp}|=\rate_G  
\end{align*}
  
The reverse inequality is trivial, since $F\subset G$, so $\rate_F=\rate_G$.

Now consider the map $\psi\from F\setminus E(h) \to [h]\setminus E(h)$
defined by $\psi(f):=fh\inv{f}$.
Choose minimal length representatives $e_1,\dots,e_m$ of $\gpcenter(h)\backslash E(h)$.
For each $g\in G\setminus E(h)$ there exists a $g'\in gE(h)$ such that
$d(\bp,g\H)=d(\bp,g'\!.\bp)$.
There exist $z\in\gpcenter(h)$ and $i$ such that $g'=gz e_i$.
Let $f:=g'e_i^{-1}$, so that $fh\inv{f}=gz e_i\inv{e_i}he_i\inv{e_i}\inv{z}\inv{g}=gh\inv{g}$. 
Since $e_i$ has length at most $\constc$ and
$\pi_{g\H}$ is $5\consta$--close to closest point projection, it
follows that $f\in F$, so $\psi$ is surjective.
Furthermore, $d(\bp,fh\inv{f}\!.\bp)\leqslant
2d(\bp,f.\bp)+d(\bp,h.\bp)$, by the triangle inequality.

On the other hand, $\psi$ is boundedly many-to-one, since if
$fh\inv{f}=f'h\inv{f'}$ then $f'\in fE(h)$, so $f\H=f'\H$.
By definition of $F$, we then have $d^\pi_{f\H}(\bp,f.\bp)\leqslant
\constb$ and $d^\pi_{f\H}(\bp,f'\!.\bp)\leqslant
\constb$, so $d(f.\bp,f'\!.\bp)\leqslant 2(\consta+\constb)$.
There are uniformly boundedly many such $f'$ for each $f$.

Hence, $\psi$ is a surjective, boundedly-many-to-one map such
that $d(\bp,\psi(f).\bp)\ladd 2d(\bp,f.\bp)$ for all $f$. 
We excluded $E(h)$ from the domain and range, but its growth exponent
is zero, since it embeds quasi-isometrically into $\X$, so $\rate_{[h]}=\rate_{[h]\setminus
  E(h)}\geqslant\frac{1}{2}\rate_{F\setminus
  E(h)}=\frac{1}{2}\rate_F=\frac{1}{2}\rate_G$. 

Now, $d^\pi_{f\H}(fh.\bp,fh\inv{f}.\bp) =
d^\pi_{f\H}(\bp,f.\bp)\leqslant \constb$ for $f\in F$, so 
 $d^\pi_{f\H}(\bp,fh\inv{f}\!.\bp)>d(f.\bp,fh.\bp)-2(\consta+\constb)$.
If $d(\bp,h.\bp)>15\consta+2\constc=\consta+2(\consta+\constb)$  then
we have
 $d^\pi_{f\H}(\bp,fh\inv{f}\!.\bp)>\consta$, so
by strong constriction, $d(\bp,fh\inv{f}\!.\bp)\geqslant
2d(\bp,f.\bp)+d(\bp,h.\bp)-4(\consta+\constb)$.
Thus,
$d(\bp,\psi(f).\bp)\eadd 2d(\bp,f.\bp)$ and $\rate_{[h]}=\frac{1}{2}\rate_{G}$.
\end{proof}

\part{Examples of Actions with Strongly Contracting Elements}\label{part:sc}
\section{Actions on Relatively Hyperbolic Spaces}\label{sec:relhyp}
Yang \cite{Yan13} proved that the action of a finitely generated group $G$ with a
non-trivial Floyd boundary on any of its Cayley graphs is growth
tight. 
Relatively hyperbolic groups have non-trivial Floyd boundaries by a
theorem of Gerasimov \cite{Ger12}, so the action of a relatively
hyperbolic group on any of its Cayley graphs is growth tight. 
It is an open question whether there exists a group with a non-trivial
Floyd boundary that is not relatively hyperbolic. 

There is also a notion of relative hyperbolicity of metric spaces,
which we will review in \fullref{sectionrelativehyperbolicity}. 
One motivating example of a relatively hyperbolic metric space is a
Cayley graph of a relatively hyperbolic group.
Another is the
universal cover $\tilde{M}$ of a complete, finite volume hyperbolic
manifold $M$. 
The fundamental group $\pi_1(M)$ of such a manifold is a relatively
hyperbolic group, so the action of $\pi_1(M)$ on any of its Cayley
graphs is growth tight by Yang's theorem.
However, this does not tell us whether the action of $\pi_1(M)$ on
$\tilde{M}$ is growth tight. 
This question was addressed for a more general class of manifolds by
Dal'bo, Peign\'e, Picaud, and Sambusetti \cite{DalPeiPic11}, who proved
growth tightness results for geometrically
finite Kleinian groups.
Using our main theorems, \fullref{maincge} and \fullref{mainqc},  we generalize their results to all groups acting on relatively hyperbolic metric spaces.

\subsection{Relatively Hyperbolic Metric Spaces}\label{sectionrelativehyperbolicity}
\begin{definition}[{cf. \cite{Dru09, Sis12}}]\label{def:relativelyhyperbolic}
  Let $\X$ be a geodesic metric space and let
  $\pers$ be a collection of uniformly coarsely
  connected subsets of $\X$. 
We say $\X$ is \emph{hyperbolic relative to the peripheral sets} $\pers$
if the following conditions are satisfied:
\begin{enumerate}
\item For each $A$ there exists a $B$ such that
  $\diam({\clnbhd{A}{\per_0}}\cap{\clnbhd{A}{\per_1}})\leqslant
B$ for distinct
$\per_0,\,\per_1\in\pers$.\label{item:boundedintersection}
\item There exists an $\epsilon\in (0,\frac{1}{2})$ and $M\geqslant 0$ such
  that if $x_0,\,x_1\in\X$ are points such that for some
  $\per\in\pers$ we have
  $d(x_i,\per)\leqslant \epsilon\cdot d(x_0,x_1)$ for each $i$, then
  every geodesic from $x_0$ to $x_1$ intersects
  ${\clnbhd{M}{\per}}$.
\item There exist $\sigma$ and $\delta$ so that for every geodesic
  triangle either:
  \begin{enumerate}
  \item there exists a ball of radius $\sigma$ intersecting all three
    sides, or\label{inradius}
\item there exists a $\per\in\pers$ such that
  ${\clnbhd{\sigma}{\per}}$ intersects all three sides and for
  each corner of the triangle, the points of the outgoing geodesics
  from that corner which first enter ${\clnbhd{\sigma}{\per}}$
  are distance at most $\delta$ apart.
  \end{enumerate}
\end{enumerate}

We say $\X$ is \emph{hyperbolic} if it hyperbolic relative to $\pers=\emptyset$.
\end{definition}

If $\X$ is hyperbolic in the sense of
\fullref{def:relativelyhyperbolic} then the only non-trivial condition
is \ref{inradius}, which is equivalent to the usual definition of
hyperbolic metric space.

\begin{definition}\label{def:relativelyhyperbolicgroup}
  A group $G$ is \emph{hyperbolic relative to a collection of finitely
  generated peripheral subgroups} if a
  Cayley graph of $G$ is hyperbolic relative to the cosets of the
  peripheral subgroups.
\end{definition}

  Sisto \cite{Sis12} shows \fullref{def:relativelyhyperbolicgroup} is
  equivalent to Bowditch's \cite{Bow12} definition of relatively
  hyperbolic groups.

\begin{definition}[cf. \cite{GroMan07}]\label{definitioncobinatorialhoroball}
Let $\X$ be a connected graph with edges of length bounded below. 
  A \emph{combinatorial horoball} based on $X$ with parameter $a>0$ is
  a graph whose vertex set is $\mathrm{Vert}\X\times
(\{0\}\cup\mathbb{N})$, contains an edge of length 1 between $(v,n)$ and
$(v,n+1)$ for all $v\in\mathrm{Vert}\X$ and all
$n\in\{0\}\cup\mathbb{N}$, and for each edge $[v,w]\in\X$ contains an edge $[(v,n),(w,n)]$
of length $e^{-an}\cdot\mathrm{length}([v,w])$.
\end{definition}

  Let $\X$ be hyperbolic relative to $\pers$.
An \emph{augmented space} is a space obtained from $\X$ as follows.
By definition, there exists a constant $C$ such that each
$\per\in\pers$ is $C$--coarsely connected.
For each $\per\in\pers$ choose a maximal subset of points that
pairwise have distance at least $C$ from one another. 
Let these points be the vertex set of a graph. 
For edges, choose a geodesic connecting each pair of vertices at
distance at most $2C$ from each other. 
Use this graph as the base of a combinatorial horoball with parameter $a_{\per}>0$.
The augmented space is the space obtained from the union of $\X$ with
horoballs $\X_\per$ for each $\per\in\pers$ by identifying the base of $\X_\per$ with the graph constructed in $\per$.

\begin{definition}\label{def:cuspuniform}
  Let $\X$ be a hyperbolic $G$--space, and let $\parabs$ be the
  collection of maximal parabolic subgroups of $G$.
 Suppose there exists a $G$--invariant collection of disjoint open
 horoballs centered at the points fixed by the parabolic subgroups. 
The \emph{truncated space} is $\X$ minus the union of these open
horoballs. 
We say $G\act \X$ is \emph{cusp uniform} if $G$ acts cocompactly
on the truncated space.
\end{definition}

If $G$ acts cocompactly on a $G$--space $\X'$ that is hyperbolic
relative to a $G$--invariant peripheral system $\pers$, then an augmented
space $\X$ can be constructed $G$--equivariantly, and $G\act X$ will be a
cusp uniform action.

Several different versions of the following theorem occur in the
literature on relatively hyperbolic groups:
\begin{theorem}[\cite{Bow12,GroMan08,Sis12}]\label{theoremrelativelyhyperbolic}
  If $\X$ is hyperbolic relative to $\pers$ then any augmented space with
  horoball parameters bounded below is hyperbolic.

If $G\act \X$ is a cusp uniform action then $G$ is hyperbolic relative
to the maximal parabolic subgroups and the truncated space is
hyperbolic relative to boundaries of the deleted horoballs.
\end{theorem}

\subsection{Quasi-convex Actions}\label{sec:qcrh}

\begin{theorem}\label{correlativelyhyperbolic}
  If $\X$ is a quasi-convex, relatively hyperbolic $G$--space and $G$
  does not coarsely fix a peripheral subspace then $G\act \X$ is a
  growth tight action.
\end{theorem}
\begin{proof}
It follows from \cite[Lemma~5.4]{Sis12} that 
  every infinite order element of $G$ that does not coarsely fix a
  peripheral subspace  is
  strongly constricting. We
  conclude by \fullref{mainqc}.
\end{proof}

\fullref{correlativelyhyperbolic} unifies the
existing proofs of growth tightness for cocompact actions on
hyperbolic spaces \cite{Sab13} and for actions of a relatively
hyperbolic group on its Cayley graphs \cite{Yan13}, 
and extends to actions on a
more general class of spaces. 

\begin{corollary}\label{infiniteends}
  The action of a finitely generated group $G$ with infinitely many ends on any one of its Cayley graphs is growth tight.
\end{corollary}
\begin{proof}
  Stallings' Theorem \cite{Sta71} says that $G$ splits non-trivially over a finite subgroup.
$G$ is hyperbolic relative to the factor groups of this splitting.
Since the splitting is non-trivial, $G$ does not fix any factor group, so
the result follows from \fullref{correlativelyhyperbolic}.
\end{proof}

\fullref{infiniteends} generalizes a result of
Sambusetti~\cite[Theorem~1.4]{Sam02amalgam}, who proved it with
additional constraints on the factor groups. 

\subsection{Cusp Uniform Actions}\label{sectioncuspuniform}
\fullref{correlativelyhyperbolic} and
\fullref{theoremrelativelyhyperbolic} show that if $G\act \X$ is a cusp
uniform action on a hyperbolic space then the action of $G$ on the truncated space is a growth
tight action.
A natural question is whether $G\act \X$ is a growth tight action. 
This action is not quasi-convex if the parabolic subgroups are
infinite, as geodesics in $\X$ will travel deeply into horoballs, and,
indeed, an example of Dal'bo, Otal, and Peign\'e \cite{DalOtaPei00} shows $G\act \X$ need not be growth tight.

To see how growth tightness can fail, consider the combinatorial
horoball from  \fullref{definitioncobinatorialhoroball}.
If $\X$ is, say, the Cayley graph of some group and we build the
combinatorial horoball with parameter $a>0$ based on $\X$, then the
$r$--ball about a basepoint $\bp\in\X$ in the horoball metric
intersected with $\X\times\{0\}$ contains
the ball of radius $C\cdot\exp(\frac{ar}{2})$ in the $\X$--metric, for
a constant $C$ depending only on $a$. 
Thus, if the number of vertices of balls in $\X$ grows faster than
polynomially in the radius, then the growth exponent with respect to
the horoball metric will be infinite. 
Furthermore, even if growth in $\X$ is polynomial we can make the
growth exponent in the horoball be as large as we like by taking $a$
to be sufficiently large. 
Dal'bo, Otal, and Peign\'e construct non-growth tight examples of
relatively hyperbolic groups with two cusps by taking one of the
horoball parameters to be large enough so that the corresponding
parabolic subgroup dominates the growth of the group; that is, the
growth exponent of the parabolic subgroup is equal to the growth
exponent of the whole group. 
Quotienting by the second parabolic subgroup then does not decrease
the growth exponent, so this action is not growth tight.

Not only does this provide an example of a non-growth tight action on
a hyperbolic space, but since augmented spaces with different horoball
parameters are equivariantly quasi-isometric to each other, we
have:
\begin{observation}\label{obs:notqi}
  Growth tightness is not invariant
among equivariantly quasi-isometric $G$--spaces.
\end{observation}

Dal'bo, Peign\'e, Picaud, and Sambusetti \cite[Theorem~1.4]{DalPeiPic11} show that this is essentially
the only way that growth tightness can fail for cusp uniform actions.
Their proof is for geometrically finite Kleinian groups, but our
\fullref{maincge} generalizes this result.

\begin{definition}\label{def:pgc}
Let $G\act\X$ be a cusp uniform action on a hyperbolic space.
Let $\parabs$ be the collection of maximal parabolic subgroups of $G$.
Then $G\act\X$ satisfies the \emph{Parabolic Gap Condition} if $\rate_\parab<\rate_G$ for all $\parab\in\parabs$.
\end{definition}

\begin{theorem}\label{theoremcuspuniform}
 Let $G$ be a finitely generated, non-elementary group.
  Let $G\act \X$ be a cusp uniform action on a hyperbolic space. 
Suppose that $G\act\X$ satisfies the Parabolic
Gap Condition.
Then $G\act \X$ is a growth tight action.
\end{theorem}
\begin{proof}
  Let $Q$ be the diameter of the quotient of the truncated space.
The $Q$--complementary growth exponent is the maximum of the parabolic
growth exponents, which, by the Parabolic Gap Condition, is strictly
less than the growth exponent of $G$. Apply \fullref{maincge}.
\end{proof}

\begin{theorem}
  Let $G$ be a finitely generated group hyperbolic relative to a 
  collection $\parabs$ of virtually nilpotent subgroups.
Then there exists a hyperbolic $G$--space $\X$ such that $G\act \X$ is
cusp uniform and growth tight.
\end{theorem}
\begin{proof}
Construct $\X$ as an augmented space by taking a Cayley graph for $G$ and attaching
combinatorial horoballs to the cosets of the peripheral subgroups. 
Since the peripheral groups are virtually nilpotent, they have
polynomial growth in any word metric \cite{Gro81}. 
It follows that the growth exponent of each parabolic group with
respect to the horoball metric is
bounded by a multiple of the horoball parameter.
By choosing the horoball parameters small enough, we can ensure
$G\act\X$ satisfies the Parabolic Gap Condition, and apply \fullref{theoremcuspuniform}.
\end{proof}

\subsection{Non Relative Hyperbolicity}\label{sec:nrh}
In subsequent sections we provide further examples of growth tight
actions. 
To show these are not redundant we will verify that the groups are not
relatively hyperbolic. 

In this section we recall a technique for showing that a group is not
relatively hyperbolic, due to 
Anderson, Aramayona, and Shackleton \cite{AndAraSha07}.
Another approach to non relative hyperbolicity, contemporaneous to and
more general than \cite{AndAraSha07}, and also implying \fullref{thm:nonrelativelyhyperbolicgog}, was developed by
Behrstock, Dru\c{t}u, and Mosher \cite{BehDruMos09}.

\begin{theorem}[{\cite[Theorem~2]{AndAraSha07}}]\label{thm:nonrelativelyhyperbolic}
  Let $G$ be a finitely generated, non-elementary group, and let
  $S$ be a (possibly infinite) generating set consisting of infinite
  order elements.
Consider the `commutativity graph' with one vertex for each element of
$S$ and an edge
between vertices $s$ and $s'$ if some non-trivial powers of $s$ and
$s'$ commute. 
If this graph is connected and there is at least one pair $s,\, s'\in
S$ such that $\langle s,s'\rangle$ contains a rank 2 free abelian
subgroup, then $G$ is not hyperbolic relative to any finite collection
of proper finitely generated subgroups.
\end{theorem}

To prove this theorem, one shows that the subgroup generated by $S$ is contained in one
of the peripheral subgroups. 
Since $S$ generates $G$ this gives a contradiction, because the
peripheral subgroups are proper subgroups of $G$.

We will actually use a mild generalization of
\fullref{thm:nonrelativelyhyperbolic} to the case when $S$ generates a
proper subgroup of $G$:

\begin{theorem}\label{thm:nonrelativelyhyperbolicgog}
Let $G$ be a finitely generated, non-elementary group.
  Let $S$ be a set of infinite order elements whose commutativity
  graph is connected and such that there is a pair $s,s'\in S$ such
  that $\langle s,s'\rangle$ contains a rank 2 free abelian subgroup.
Consider the `coset graph' whose vertices are cosets of $\langle S\rangle$,
with an edge connecting $g\langle S\rangle$ and $h\langle S\rangle$ if
$g\langle S\rangle\inv{g}\cap h\langle S\rangle\inv{h}$ is infinite.
If this graph is connected, then $G$ is not hyperbolic relative to any finite collection
of proper finitely generated subgroups.
\end{theorem}
\begin{proof}
  Suppose $G$ is hyperbolic relative to $\{\parab_1,\dots,\parab_k\}$.
As in the proof of \fullref{thm:nonrelativelyhyperbolic}, $\langle S\rangle$ is
contained in a conjugate of some $\parab_i$.
We assume, without loss of generality, that $\langle S\rangle\subset \parab_1$. 
Condition (\ref{item:boundedintersection}) of
\fullref{def:relativelyhyperbolic} implies $\parab_i\cap g\parab_i\inv{g}$ is finite for
$g\notin \parab_i$.
Thus, for $g\langle S\rangle$ adjacent to $\langle S\rangle$ in the
coset graph, $g\in \parab_1$ and $g\langle S\rangle\inv{g}\subset\parab_1$.
Connectivity of the coset graph implies that every
element of $G$ is contained in $\parab_1$, contradicting the
hypothesis that $\parab_1$ is a proper subgroup.
\end{proof}

We also note that \fullref{thm:nonrelativelyhyperbolic}
and \fullref{thm:nonrelativelyhyperbolicgog} imply the, a priori,
stronger result that $G$ has trivial Floyd boundary.

\section{Rank 1 Actions on CAT(0) Spaces}\label{sec:cat0}
A metric space is \emph{CAT(0)} if every geodesic triangle is at least
as thin as a triangle in Euclidean space with the same side lengths.
An isometry $\phi$ of a CAT(0) space $\X$ is \emph{hyperbolic} if
$\inf_{x\in\X}d(x,\phi(x))$ is positive and is attained.
See, for example, \cite{BriHae99} for more background.

Let $\X$ be a CAT(0) $G$--space. 
Recall that our definition of `$G$--space' includes the hypothesis that $\X$ is proper, so
an element is strongly contracting if and only if it acts as a
rank 1 isometry:
\begin{theorem}[{\cite[Theorem~5.4]{BesFuj09}}]\label{Thm:rank1}
Let $h$ be a hyperbolic isometry of a proper CAT(0) space $\X$ with
axis $\A$.
Closest point projection to $\A$ is strongly contracting if and only if $\A$ does not bound an isometrically embedded half-flat in $\X$.
\end{theorem}

\fullref{Thm:rank1} and \fullref{mainqc} show:
\begin{theorem}\label{cat0}
  If $G$ is a non-elementary, finitely generated group and $\X$ is a quasi-convex, CAT(0)
  $G$--space such that $G$ contains
an element that acts as a rank 1 isometry on $\X$, then $G\act \X$ is a growth tight action.
\end{theorem}

Recall from \fullref{sec:rankrigidity} that there are many interesting
classes of CAT(0) spaces that admit rank 1 isometries.
In the remainder of this section we highlight a few examples.

Let $\Theta$ be a simple graph.
The \emph{Right Angled Artin Group $G(\Theta)$ defined by $\Theta$} is the group
defined by the presentation $\langle g_v \text{ for }
v\in\mathrm{Vert}(\Theta)\mid g_vg_wg_v^{-1}g_w^{-1}=1\text{ for }
[v,w]\in\mathrm{Edge}(\Theta)\rangle$.
The graph $\Theta$ also determines a cube complex constructed by
taking a rose with one loop for each vertex of $\Theta$, and then
gluing in a $k$--cube to form a $k$--torus for each complete $k$--vertex subgraph of
$\Theta$. The resulting complex is called the \emph{Salvetti complex},
and its fundamental group is 
$G(\Theta)$.
The universal cover of the Salvetti complex turns out to be a CAT(0)
cube complex.
See \cite{Cha07} for more background on Right Angled Artin Groups.

If $\Theta$ is a single vertex then $G(\Theta)\cong\mathbb{Z}$ is elementary.
If $\Theta$ is a join, that is, if it is a complete bi-partite graph,
then $G(\Theta)$ is a direct product of Right Angled Artin Groups
defined by the two parts.
In all other cases, we find a growth tight action:

\begin{theorem}\label{raag}
Let $\Theta$ be a finite simple graph that is not a join and has more than
one vertex.
The action of the Right Angled Artin Group $G(\Theta)$ defined by $\Theta$ on
the universal cover $\X$ of the
Salvetti complex associated to $\Theta$ is a growth tight action.
\end{theorem}
\begin{proof}
The universal cover $\X$ of the Salvetti complex of $\Theta$ is a
cocompact, CAT(0) $G(\Theta)$--space.
If $\Theta$ is not connected then $\X$ is hyperbolic relative to
subcomplexes defined by the components of $\Theta$, so $G(\Theta)\act\X$ is
growth tight by \fullref{correlativelyhyperbolic}.
If $\Theta$ is connected then $G(\Theta)$ contains a rank 1 isometry by a
theorem of Behrstock and Charney \cite{BehCha12}.
The result follows from \fullref{cat0}.
\end{proof}

The defining graph of a Right Angled Artin Group is a
commutativity graph. If this graph is connected then the group is not relatively hyperbolic by
\fullref{thm:nonrelativelyhyperbolic}.

\bigskip

A \emph{flip-graph-manifold} is a compact three dimensional manifold
$M$ with boundary obtained from a finite collection of Seifert fibered
pieces that are each a product of 
a circle with a compact oriented hyperbolic surface with boundary.
These are glued together along boundary tori by a map exchanging the fiber and base
directions. 
Such manifolds were studied by Kapovich and Leeb \cite{KapLee98}, who
show that the universal cover of $M$ admits a CAT(0) metric, and that
an element of $\pi_1(M)$ that acts hyperbolically is rank 1 if and only if it
is not represented by a loop contained in a single Seifert fibered piece.
Thus, \fullref{cat0} implies:

\begin{theorem}\label{flip}
The action of the fundamental group of a flip-graph-manifold by deck
transformations on its universal cover with its natural CAT(0) metric is a growth tight action.
\end{theorem}

To see that the fundamental group of a flip-graph-manifold is
not-relatively hyperbolic, apply
\fullref{thm:nonrelativelyhyperbolicgog} where $S$ is the set of
elliptic elements for the action of $G$ on the Bass-Serre tree of the
defining graph of groups decomposition.

\fullref{raag} and \fullref{flip} give the first non-trivial
examples of growth tight actions on spaces that are not relatively hyperbolic.

The idea of the proof for flip-graph-manifolds generalizes to other CAT(0)
graphs of groups via \fullref{corollary:gog}:
\begin{theorem}\label{cat0gog}
Let $G$ be a non-elementary, finitely generated group that splits
non-trivially as a
graph of groups and is not an ascending HNN-extension.
Suppose that the corresponding action of $G$ on the Bass-Serre tree of
the splitting has two edges whose stabilizers have finite intersection.
Suppose there exists a cocompact, CAT(0) $G$--space $\X$.
Then $G\act\X$ is a growth tight action.
\end{theorem}
\begin{proof}
By \fullref{corollary:gog}, $G$ contains an infinite order
element $h$ such that $E(h)$ is hyperbolically embedded.
A theorem of Sisto \cite{Sis13} implies that any axis of $h$ is a
Morse quasi-geodesic.
An element with an axis that bounds a half-flat is not Morse, so $h$
is rank 1, and 
the result follows by \fullref{cat0}.
\end{proof}

\section{Mapping Class Groups}\label{sec:mcg}

Let $\s=\s_{g,p}$ be a connected and oriented surface of genus $g$ with
  $p$ punctures. We require \s to have negative Euler characteristic. 
  
  Given two orientation-preserving homeomorphisms $\phi, \psi \colon \s \to
  \s$, we will consider $\phi$ and $\psi$ to be equivalent if $\phi \circ
  \psi^{-1}$ is isotopic to the identity map on \s. Each equivalence class
  is called a \emph{mapping class} of \s, and the set \mcg of all
  equivalence classes naturally forms a group called the \emph{mapping
  class group} of \s.   

  A mapping class $f \in \mcg$ is called \emph{reducible} is there exists
  an $f$--invariant curve system on \s and \emph{irreducible} otherwise. By
the  Nielsen-Thurston classification of elements of \mcg, a mapping class is
  irreducible and infinite order if and only if it is pseudo-Anosov
  \cite{Thu86b}.  

Let $\X$ be the Teichm\"uller space of marked hyperbolic structures on $\s$, equipped with the
Teichm\"uller metric. See \cite{Hub06} and \cite{Pap07} for more information.

  \begin{theorem}[\cite{Min96}]\label{thm:minsky}
    Every pseudo-Anosov element is strongly contracting for $\mcg\act\X$.
  \end{theorem}

For each $\epsilon>0$ there is a decomposition of $\X$ into a `thick
part' $\X^{\geqslant \epsilon}$ and a `thin part' $\X^{<\epsilon}$
according to whether the hyperbolic structure on $\s$ corresponding to
the point $x\in\X$ has any closed curves of length $<\epsilon$. 
This decomposition is $\mcg$--invariant, and
$\mcg\act\X^{\geqslant\epsilon}$ is cocompact, see \cite{Mum71} and
\cite{FarMar12}.
Geodesics between points in the thick part can travel deeply into the
thin part, so the action of $\mcg$ on Teichm\"uller space is not quasi-convex. 
To prove growth tightness need a bound on the
complementary growth exponent. 
Such a bound is provided by a recent theorem of Eskin, Mirzakhani, and Rafi \cite[Theorem~1.7]{EskMirRaf12}.

\begin{theorem}\label{mcgisgrowthtight}
  The action of the mapping class group $\mcg$ of $\s=\s_{g,p}$ on its
  Teichm\"uller space $\X$ with
  the Teichm\"uller metric is a growth tight action. 
\end{theorem}
\begin{proof}
  Let $\zeta=6g-6+2p\geqslant 2$. 
The growth exponent of $\mcg$ with respect to its action on $\X$ is
$\zeta$ \cite{AthBufEsk12}. (We remark that the result of \cite{AthBufEsk12} is stated for closed surfaces, but their proof works in general. For our interest, it is enough that the growth exponent of $\mcg$ is bounded below by $\zeta$. This can be obtained from \cite{Ham13} and \cite{EskMirRaf12}.)

Choose an $r_0$ and a maximal $r_0$--separated set in moduli space
$\mcg\backslash\X$, and let $\A$ be its full lift to $\X$.
Given $r_0$ as above and $\delta=\frac{1}{2}$, let $\epsilon$ be
sufficiently small as in
\cite[Theorem~1.7]{EskMirRaf12}.
Let $Q$ be the smallest number such
that the $\epsilon$--thick part of $\X$ is contained in $\clnbhd{Q}{\mcg.\bp}$.
Choose a finite subset $\{a_1,\dots,a_n\}\subset \A$ such that:
\[\clball{Q}{\bp}\setminus\nbhd{Q}{\mcg.\bp}\subset\bigcup_{i=1}^n\ball{r_0}{a_i}\]

Suppose that $g\in\mcg$ is such that there exists a geodesic $[x,y]$ between
$\clball{Q}{\bp}$ and $\clball{Q}{g.\bp}$ whose interior stays
in $\X\setminus\clnbhd{Q}{\mcg.\bp}$.
Then there are indices $i$ and $j$ such that $x\in \ball{r_0}{a_i}$ and
$y\in \ball{r_0}{g.a_j}$. 
This means that every element contributing to  $Comp^{\mcg}_{Q,\,r}$ of
\fullref{def:cge} also contributes to some
$N_1(\mathcal{Q}_{1,\epsilon},a_i,a_j,r)$ of
\cite[Theorem~1.7]{EskMirRaf12}.
The conclusion of \cite[Theorem~1.7]{EskMirRaf12} is that 
$N_1(\mathcal{Q}_{1,\epsilon},a_i,a_j,r)\leqslant
G(a_i)G(a_j)\exp(r\cdot(\zeta-\frac{1}{2}))$ for all sufficiently
large $r$, where $G$ is a particular function on $\X$.
There are finitely many such sets, and the function $G$ is bounded on
$\{a_1,\dots,a_n\}$, so there is a constant $C$ such that
$Comp^{\mcg}_{Q,\,r}\leqslant C \cdot\exp(r\cdot(\zeta-\frac{1}{2}))$ for all sufficiently
large $r$.
Thus, the $Q$--complementary growth exponent is at most 
$\zeta-\frac{1}{2}<\zeta$.
The theorem now follows from \fullref{thm:minsky} and \fullref{maincge}.
\end{proof}

  When the genus of $\s$ is at least 3 then there does
not exist a cocompact, CAT(0) $\mcg$--space \cite{Bri10}.
The fact that such an $\mcg$ is not relatively hyperbolic (in fact, has
trivial Floyd boundary) is an application of
\fullref{thm:nonrelativelyhyperbolic} appearing in \cite{AndAraSha07}.
Therefore, \fullref{mcgisgrowthtight} does not follow from
the results of the previous sections.

A natural question is whether the action of a mapping class group on
its Cayley graphs is growth tight. 
There is also a combinatorial model for the mapping class group known
as the \emph{marking complex}.
Finally, a mapping class group acts cocompactly on a thick part of
Teichm\"uller space.
All of these spaces are quasi-isometric, and 
 Duchin and Rafi \cite{DucRaf09} show that pseudo-Anosov elements are
 contracting for the action of a mapping class group on any one of
 its Cayley graphs, but we do not know whether one of these actions
 admits a strongly contacting element.
 \begin{question}
   Is the action of a mapping class group of a hyperbolic surface on one of its Cayley graphs/marking
   complex/thick part of Teichm\"uller space growth tight?
 \end{question}

The outer automorphism group of a finite rank non-abelian free group,
$\mathrm{Out}(F_n)$ is often studied in analogy with $\mcg$.
Algom-Kfir \cite{Alg11} has proven an analogue of Minsky's theorem
that says that a \emph{fully irreducible} outer automorphism class
is strongly contracting for the action of $\mathrm{Out}(F_n)$ on its
Outer Space, which is the analogue of the Teichm\"uller space. 
However, we lack the analogue of the Eskin-Mirzakhani-Rafi theorem
that was used to control the complementary growth exponent in the
mapping class group case.
 
There is also an analogue of the thick part of Teichm\"uller space
called the \emph{spine} of the Outer Space, on which
$\mathrm{Out}(F_n)$ acts cocompactly.

\begin{question}
  Is the action of $\mathrm{Out}(F_n)$ on one of its Cayley
  graphs/Outer Space/spine of Outer Space growth tight?
\end{question}

\section{Snowflake Groups}\label{sec:snowflake}
Let $G:=BB(1,r)=\langle a,b,s,t\mid ab\inv{a}\inv{b}=1,\, \inv{s}as=a^rb,\,\inv{t}at=a^r\inv{b}\rangle $
be a Brady-Bridson snowflake group with $r\geqslant 3$. 
Let $L:=2r$. 
These groups have an interesting mixture of positive and negative
curvature properties.
$G$ splits as an amalgam of $\mathbb{Z}^2=\langle a,b\rangle $ by two
cyclic groups $\langle a^rb\rangle $ and $\langle a^r\inv{b}\rangle $, and the
action of $G$ on the Bass-Serre tree $\mathcal{T}$ of this splitting
satisfies \fullref{corollary:gog}, so $G$ has hyperbolically embedded subgroups.
However, we can not automatically conclude that such a hyperbolically
embedded subgroup gives rise to a strongly contracting element, as there does not exist a
cocompact, CAT(0) $G$--space.
If such a space existed, then the Dehn function
of $G$ would be at most quadratic, but Brady and Bridson
\cite{BraBri00} have shown that the Dehn function of $BB(1,r)$ is $n^{2\log_2 L}>n^2$. 

We will fix a $G$--space $\X$ and demonstrate two different elements of $G$ that act
hyperbolically on $\mathcal{T}$ such that the pointwise stabilizer of any
length 3 segment of their axes is finite.
One of these elements will be strongly contracting for the action on $\X$, and the other will
not.
Hence:
\begin{theorem}\label{snowflakegrowthtight}
  $G$ admits a cocompact growth tight action.
\end{theorem}

Observe that
\fullref{thm:nonrelativelyhyperbolicgog} with $S:=\{a,b\}$ shows that $G$
is not relatively hyperbolic.

\subsection{The Model Space $\X$}
Let $\X$ be the Cayley graph for $G$ with respect to the generating
set $\{a,\, a^rb,\, a^r\inv{b},s,t\}$, where the edges corresponding
to $a^rb$ and  $a^r\inv{b}$ have been rescaled to have length $L:=2r$.
The point of scaling these edges is that $a^rb$, $a^r\inv{b}$, and
$a^{2r}$ form an equilateral triangle of side length $L$, which will
facilitate finding geodesics in this particular model.

It is also useful to consider $G$ as the fundamental group of the topological space obtained from a torus by gluing on two annuli.
Choose a basepoint for the torus and for each boundary component of the annuli.
For one annulus, the $s$--annulus,  glue the two boundary curves to the curves $a$ and $a^rb$ in the torus, gluing basepoints to the basepoint of the torus.
For the other annulus, the $t$--annulus, glue the two boundary curves to the curves $a$ and $a^r\inv{b}$ of the torus.
The resulting space is a graph of spaces \cite{ScoWal79} associated to
the given graph of groups decomposition of $G$.
The fundamental group of this space is $G$, which acts freely by deck
transformations on the universal cover $\X'$. 
Choose the basepoint $\bp$ of $\X'$ to be a lift of the basepoint of the torus. 
The correspondence between a vertex $g\in \X$ and the point $g.\bp\in\X'$ inspires the following terminology:
A \emph{plane} is a coset $g\langle a,b\rangle \in G/\langle a,b\rangle $, which corresponds to a lift of the torus at the point $g.\bp\in\X'$.
An $s$-\emph{wall} is the set of outgoing $s$--edges incident to a
coset $g\langle a\rangle  \in G/\langle a\rangle$. 
This corresponds to a lift of the $s$--annulus at the point $g.\bp\in\X'$.
A $t$-\emph{wall} is the set of outgoing $t$--edges incident to a
coset $g\langle a\rangle  \in G/\langle a\rangle $. 
This corresponds to a lift of the $t$--annulus at the point $g.\bp\in\X'$.
 Each wall separates $\X$ (and $\X'$) into two complementary components.
Notice that the origins of consecutive edges in an $s$--wall are connected by a single $a$--edge of length $1$, while the termini of those edges are connected by a single $a^rb$--edge of length $L$. 
We say that crossing an $s$--wall in the positive direction scales distance by a factor of $L$.
The same is true for the $t$--walls.

\subsection{Geodesics Between Points in a Plane}\label{geodesics}
We will define a family of $\X$--geodesics joining $1$ to every point of $\langle a,b\rangle $.
This is similar to the argument of \cite{BraBri00}.

From the fact that $\langle a,b\rangle $ is abelian, for every point $a^xb^y$ there is a geodesic from $1$ to $a^xb^y$ of the form:
\[[1,(a^rb)^m]+(a^rb)^m[1,(a^r\inv{b})^n]+(a^rb)^m(a^r\inv{b})^n[1,a^p]\]
where $[g,h]$ indicates a geodesic from $g$ to $h$.

For a point of the form $(a^rb)^m$ there is an $a^rb$--edge path from
$1$ to $(a^rb)^m$ of length $mL$.
This path is clearly inefficient, as it lies along the boundary of an
$s^{-1}$--wall that scales distance by $1/L$, so we can push the original edge
path across the wall to a path $s^{-1}a^ms$ of length $2+m$.
We claim there is a geodesic from $1$ to $(a^rb)^m$ of the form $[1,\inv{s}]+\inv{s}[1,a^m]+\inv{s}a^m[\inv{s},1]$.
We have already exhibited a wall-crossing path of length $2+m$, which is shorter than any path from $1$ to $(a^rb)^m$ that stays in the plane $\langle a,b\rangle$.
Thus, a geodesic must cross some walls. 
Every path from $1$
to $(a^rb)^m$ can, by rearranging subsegments and eliminating backtracking, be replaced by a path
of at most the same length and having the form
$\gamma_s+\gamma_t+\gamma'$ where:
\begin{itemize}
\item $\gamma_s=[1,\inv{s}]+\inv{s}[1,a^n]+\inv{s}a^n[\inv{s},1]$ if non-trivial.
  \item $\gamma_t=\inv{s}a^ns[1,\inv{t}]+\inv{s}a^ns\inv{t}[1,a^p]+\inv{s}a^ns\inv{t}a^p[\inv{t},1]$
  if non-trivial.
\item $\gamma'=\inv{s}a^ns\inv{t}a^pt[1,a^q]$ if non-trivial.
\end{itemize}
 
The path $\gamma=\gamma_s+\gamma_t+\gamma'$ is a path from $1$ to $\inv{s}a^ns\inv{t}a^pta^q=(a^rb)^n(a^r\inv{b})^p a^q=a^{r(n+p)+q}b^{n-p}=a^{rm}b^m$, so $p=n-m$ and $q=-Lp$.
Since $p$ and $q$ are proportional, $\gamma_t$ and $\gamma'$ are either
both trivial or both non-trivial. Suppose they are non-trivial. 
There is a symmetry that exchanges $\gamma_t$ with a path
$\gamma_t'=\inv{s}a^ns[1,\inv{s}]+\inv{s}a^ns\inv{s}[1,a^{-p}]+\inv{s}a^ns\inv{s}a^{-p}[\inv{s},1]$
of the same length.
However, $\gamma_t'$ and $\gamma_t+\gamma'$ have the same endpoints,
and $\gamma_t'$ is shorter,
so $\gamma$ could not have been geodesic if $\gamma_t$ and $\gamma'$
are non-trivial. 
Thus, if $\gamma$ is geodesic then $\gamma=\gamma_s$.
This reduces the problem of finding a geodesic from $1$ to $(a^rb)^m$
to finding a geodesic from $1$ to $a^n$.

A similar argument holds for geodesics from $1$ to $(a^rb^{-1})^m$, so
we can find geodesics from 1 to any point in $\langle a,b\rangle$ if
we know geodesics from 1 to powers of $a$.

For powers of $a$ the idea is that $a^{mL}$, $(a^rb)^m$, and
$(a^r\inv{b})^m$ form an equilateral triangle in the plane, but the
latter two can be shortened by a factor of $L$ by pushing across a
wall.
Since $L\geq 6$, the savings of a factor of $\sfrac{L}{2}$ in length
outweighs the added overhead from crossing walls.

For small powers of $a$ we can find geodesics by inspection of the
Cayley graph.
For $0\leqslant |p|\leqslant \sfrac{L}{2}+3$, the edge path $a^p$ from $1$ to $a^p$ is a geodesic of length $|p|$. 
For $\sfrac{L}{2}+3\leqslant p \leqslant L$ the edge path $\inv{s}as\inv{t}ata^{p-L}$ is a geodesic from $1$ to $a^p$ of length $6+L-p$.
We conclude that for $m>0$ and $-\sfrac{L}{2}+3\leq p\leq\sfrac{L}{2}+3$ there is a geodesic from $1$ to $a^{mL+p}$ of the form:
\begin{align*}
  [1,\inv{s}]&+\inv{s}[1,a^m]+\inv{s}a^m[\inv{s},1]\\
&+\inv{s}a^ms[1,\inv{t}]+\inv{s}a^ms\inv{t}[1,a^m]+\inv{s}a^ms\inv{t}a^m[\inv{t},1]\\
&+\inv{s}a^ms\inv{t}a^mt[1,a^p]
\end{align*}

We can now find geodesics from $1$ to powers of $a$ by induction, and
from these we know a geodesic from 1 to any $a^xb^y$.
We see an example in \fullref{fig:snowflake}, where trapezoids are
walls and triangles are contained in planes.
The top half boundary and bottom half boundary of the figure each give geodesics of
length $5\cdot 2^5-4$ between $1$ and $a^{L^5}$. (This form of geodesic
loop bears witness to the Dehn function \cite{BraBri00}, and inspired the name
`snowflake group' \cite{BraBriFor09}.)

\begin{figure}[h]
  \centering
\labellist
\pinlabel $1$ [r] at 111 245
\pinlabel $a^{L^5}$ [l] at 343 245
\pinlabel $(a^rb)^{L^4}$ [b] at 233 442
\pinlabel $(a^r\inv{b})^{L^4}$ [t] at 233 43
\pinlabel $\inv{s}$ [tr] at 116 309
\pinlabel $\inv{s}a^{L^4}$ [tr] at 171 408
\pinlabel $\inv{t}$ [br] at 117 178
\pinlabel $\inv{t}a^{L^4}$ [br] at 167 80
\endlabellist
\includegraphics[width=.6\textwidth]{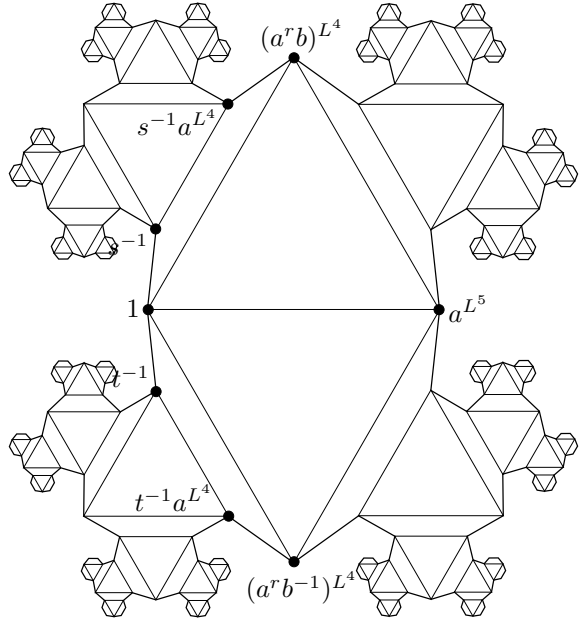}
  \caption{Snowflake - The boundary is a geodesic loop of length $2(5\cdot 2^5-4)$}\label{fig:snowflake}
\end{figure}

\subsection{Projections to Geodesics in $\X$}
In this section we consider two different geodesics:
$\alpha(2n)=(\inv{s}t)^n$ and 
$\beta(n)=s^{-n}$.
These are geodesics since for each of these paths, every edge crosses a distinct wall.
Let $\mathcal{T}$ be the Bass-Serre tree of $G$, and let $\bp\in\mathcal{T}$
be the vertex fixed by the subgroup $\langle a,b\rangle $.
The orbit map $g\mapsto g.\bp$ sends each of $\alpha$ and $\beta$ isometrically to a geodesic in $\mathcal{T}$.
We will use $\pi_\alpha$ to denote closest point projection to
$\alpha$, both in $\X$ and in $\mathcal{T}$, and similarly for $\beta$.

Both of these geodesics have the property that for any vertices at
distance at least three in the corresponding geodesic of the
Bass-Serre tree, the pointwise stabilizers of the pair of vertices is trivial.
We might hope, in analogy to \fullref{cat0gog}, that these would be
strongly contracting geodesics.
As in \fullref{cat0gog}, $\langle \inv{s}t\rangle $ and $\langle s\rangle $
are hyperbolically embedded subgroups in $G$, but, of the two, we will
see only
$\inv{s}t$ is strongly contracting.

\subsubsection{$\alpha$} 
We claim that closest point projection $\pi_\alpha\from \X\to\alpha$ is coarsely well defined and strongly contracting.
First, consider $\pi_\alpha$ on $\langle a,b\rangle $. 
The geodesic $\alpha$ enters $\langle a,b\rangle $ through the incoming $t$--wall $V$ at 1, and exits through the outgoing $\inv{s}$--wall $W$ at 1.
\begin{lemma}\label{walls}
  For every $v\in V$ and every $w\in W$ there exists a geodesic from $v$ to $w$ that includes the vertex 1.
\end{lemma}
\begin{proof}
  The lemma follows from the discussion of geodesics in \fullref{geodesics}.
\end{proof}

\begin{lemma}
  The orbit map $\X\to\mathcal{T} : g\mapsto g.\bp$ coarsely commutes
  with closest point projection to $\alpha$. In particular, closest
  point projection to $\alpha$ in $\X$ is coarsely well defined.
\end{lemma}
\begin{proof}
  Suppose $z\in\X$ is some vertex that is separated from 1 by $V$, and
  suppose there is an $n\geqslant 0$ such that $\alpha(n)\in\pi_\alpha(z)$.
Let $\sigma$ be a geodesic from $z$ to $\alpha(n)$.
Write $\sigma=\sigma_1+\sigma_2+\sigma_3$, where $\sigma_2$ is the subsegment of $\sigma$ from the first time $\sigma$ crosses $V$ until the first time $\sigma$ reaches $W$.
By \fullref{walls}, we can replace $\sigma_2$ by a geodesic segment $\sigma_2'+\sigma_2''$ where the concatenation point is 1.
This means that $z$ is connected to $1=\alpha(0)$ by a path $\sigma_1+\sigma_2'$. 
By hypothesis, the length of this path is at least the length of $\sigma$, so $\sigma_2''$ and $\sigma_3$ are trivial and $n=0$.
It follows immediately that
the orbit map $\X\to\mathcal{T}$ commutes with $\pi_\alpha$ up to
an error of 4.
(In fact, a little more work will show the error is at most 2.)
\end{proof}

\begin{lemma}[Bounded Geodesic Image Property for $\pi_\alpha$]
  For any geodesic $\sigma$ in $\X$, if the diameter of $\pi_\alpha(\sigma.\bp)$ is at least 5, then $\sigma\cap\alpha\neq \emptyset$.
\end{lemma}
\begin{proof}
 Suppose $\alpha([-1,3]).\bp\subset\pi_\alpha(\sigma.\bp)$.
Then $\sigma$ crosses the walls $V$, $W$, $\inv{s}tV$ and $\inv{s}tW$.
Write $\sigma$ as a concatenation of geodesic subsegments
$\sigma_1+\sigma_2+\sigma_3+\sigma_4+\sigma_5$, where $\sigma_1$ is
all of $\sigma$ prior to the first $V$ crossing, $\sigma_2$ is the
part of $\sigma$ between the first $V$ crossing and the last $W$
crossing, $\sigma_3$ is the part between the last $W$ crossing and the
first $\inv{s}tV$ crossing, which included 
edges labeled $\inv{s}$ and $t$, $\sigma_4$ is the part from the first
$\inv{s}tV$ crossing until the last $\inv{s}tW$ crossing, and $\sigma_5$ is the remainder of $\sigma$.
We can apply \fullref{geodesics} to replace $\sigma_2$ by a geodesic $\sigma_2'+\sigma_2''$ with the same endpoints and concatenated at 1. 
Similarly, we can replace $\sigma_4$ by a geodesic $\sigma_4'+\sigma_4''$ with the same endpoints and concatenated at $\inv{s}t$.
But then we can replace the subsegment $\sigma_2+\sigma_3+\sigma_4$ of $\sigma$ by the path $\sigma_2''+[1,\inv{s}t]+\sigma_4''$ with the same endpoints.
This path is strictly shorter unless $\sigma_2''$ and $\sigma_4''$ are trivial. 
This means that $[1,\inv{s}t]\subset\sigma\cap \alpha$.
\end{proof}
By \fullref{prop:allthesame}, this means:
\begin{corollary}
  The element $\inv{s}t$ is strongly contracting for $G\act\X$.
\end{corollary}
Together with \fullref{mainqc}, this proves \fullref{snowflakegrowthtight}.

\subsubsection{$\beta$}\label{sec:beta}
Using out knowledge of geodesics from \fullref{geodesics}, we see that
the closest point of the $s^{-1}$--wall at 1 to the point $a^{L^k}$ is
$(a^rb)^{L^{k-1}}$, which is the midpoint of a geodesic from $1$ to
$a^{L^k}$.
This geodesic coincides with $\beta$ on the
interval from 1 to $s^{-k}$.
It follows that $\pi_\beta(a^{ L^j})=\beta(j)$ for all $j\geqslant 0$.

\begin{figure}[h]
  \centering
\labellist
\small
\pinlabel $\beta(0)=s^0=1$ [r] at 110 217
\pinlabel $a^{L^3}$ [bl] at 175 215
\pinlabel $a^{L^5}$ [l] at 343 215
\tiny
\pinlabel $s^{-3}$ [l] at 76 285
\pinlabel $s^{-5}$ [l] at 85 275 
\endlabellist
 \includegraphics[width=.6\textwidth]{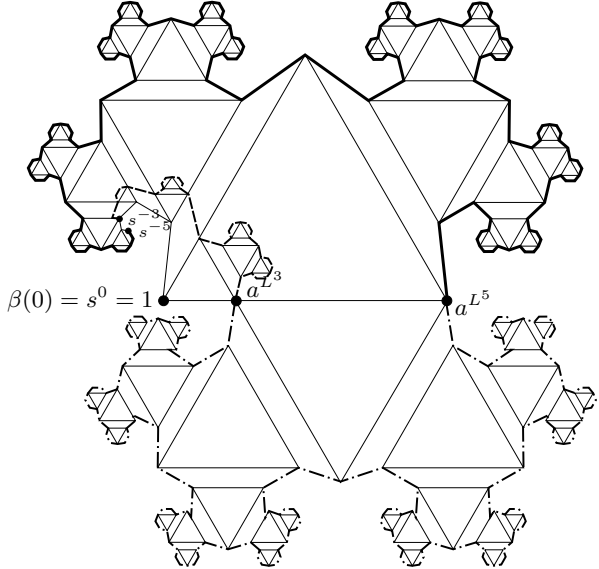}
\caption{Geodesics {$[a^{L^3},\pi_\beta(a^{L^3})]$ (dashed)},
  {$[a^{L^5},\pi_\beta(a^{L^5})]$~(solid)}, and
  {$\sigma_{3,5}=[a^{L^3},a^{L^5}]$ (dash-dot)}}\label{fig:notbgi}
\end{figure}

For $0< j<k$ there is a geodesic $\sigma_{j,k}$ from $a^{L^j}$ to
$a^{L^k}$ such that $d(\sigma_{j,k},\beta)=d(a^{L^j},\beta)$. See \fullref{fig:notbgi}.
Letting $j$ and $k-j$ grow large, the geodesics $\sigma_{j,k}$ for
stay outside large neighborhoods of $\beta$ but have large projections
to $\beta$.
Therefore, $\pi_\beta$ is not strongly contracting, since it does
not enjoy the Bounded Geodesic Image Property.



\bibliographystyle{hyperamsplain}
\bibliography{GrowthTightActions}

\end{document}